\documentclass[12pt]{article}
\usepackage{graphicx,amsmath,amsfonts,amssymb,amscd,amsthm, mathrsfs}
\newtheoremstyle{kai}
{3pt}{3pt}{}{}{\bfseries}{.}{.5em}{}
\makeatletter
\def\EquationsBySection{\def\theequation
{\thesection.\arabic{equation}}%
\@addtoreset{equation}{section}}

\newcommand\old[1]{}
\newcommand{\pend}{\hfill \thicklines \framebox(6.6,6.6)[l]{}}
\renewenvironment{proof}{\noindent {\it  Proof.} \rm}{\pend}
\newtheorem{theorem}{Theorem}[section]
\newtheorem{lemma}{Lemma}[section]
\newtheorem{corollary}{Corollary}[section]
\newtheorem{proposition}{Proposition}[section]
\newtheorem{remark}{Remark}[section]
\newtheorem{definition}{Definition}[section]

\EquationsBySection
\makeatother

\begin{document}
\pagestyle{plain}
\title
{\bf On Regularity Property of Retarded Ornstein-Uhlenbeck Processes\\  in Hilbert Spaces}
\author{
Kai Liu\\
\\
 \small{Division of Statistics and Probability,}\\
\small{Department of Mathematical Sciences,}\\
\small{The University of Liverpool,}\\
\small{Peach Street, Liverpool, L69 7ZL, U.K.}\\
\small{E-mail: k.liu@liv.ac.uk}\\}

\date{}
\maketitle

\noindent {\bf Abstract:} In this work, some regularity properties of mild solutions for a class of stochastic linear  functional differential equations driven by infinite dimensional Wiener processes are considered. In terms of retarded fundamental solutions, we introduce a class of stochastic convolutions which naturally arise in the solutions and investigate their Yosida approximants.  By means of the retarded fundamental solutions, we find conditions under which each mild solution permits a continuous modification.
With the aid of Yosida approximation, we study two kinds of regularity properties, temporal and spatial ones, for the retarded solution processes.
 By employing a factorization  method, we establish a retarded version of Burkholder-Davis-Gundy's inequality for stochastic convolutions.        


\noindent{\bf 2000 Mathematics Subject Classification(s):} 60H15, 60G15, 60H05.

\newpage

\section{Introduction}
Let $H$ and $K$ be two real separable Hilbert spaces with associated inner products $\langle\cdot, \cdot\rangle_H$,  $\langle\cdot, \cdot\rangle_K$ and norms $\|\cdot\|_H$, $\|\cdot\|_K$, respectively. We denote by ${\mathscr L}(K, H)$ the set of all linear bounded operators from $K$ into $H$, equipped with the usual operator norm $\|\cdot\|$ topology. When $H=K$, we denote ${\mathscr L}(H, H)$ simply by ${\mathscr L}(H)$.

Let $\{\Omega, {\mathscr F}, {\mathbb P}\}$ be a
complete probability space equipped with some filtration $\{\mathscr{F}_{t}\}_{t\geq 0}$ satisfying the usual
     conditions, i.e., the filtration is right continuous and
      $\mathscr{F}_{0}$ contains all ${\mathbb P}$-null sets. 
Let  $\{W(t),\,t\ge 0\}$ denote a $K$-valued $\{\mathscr{F}_{t}\}_{t\geq 0}$-Wiener process defined on  $\{\Omega, {\mathscr F}, {\mathbb P}\}$ with covariance operator $Q$, i.e.,
\[
\mathbb{E}\langle W(t), x\rangle_K\langle W(s), y\rangle_K = (t\wedge s)\langle Qx, y\rangle_K\,\,\,\,\hbox{for all}\,\,\,\,\, x,\,\,y\in K,\]
where $Q$ is a linear, symmetric and nonnegative bounded operator on $K$. In particular, we shall call $W(t)$, $t\ge 0$, a $K$-valued $Q$-Wiener process with respect to $\{\mathscr{F}_{t}\}_{t\geq 0}$. If the trace  Tr$\,Q<\infty$, then $W$ is a genuine Wiener process. It is possible that Tr$\,Q=\infty$, e.g., $Q=I$ which corresponds to a cylindrical Wiener process.

In order to define stochastic integrals with respect to the $Q$-Wiener process $W(t)$, we introduce the subspace $K_Q=\hbox{Ran}\,Q^{1/2}\subset K$, the range of $Q^{1/2}$, which is a Hilbert space endowed with the inner product
\[
\langle u, v\rangle_{K_Q} =\langle Q^{-1/2}u, Q^{-1/2}v\rangle_K\hskip 10pt\hbox{for any}\hskip 10pt  u,\,\,v\in K_Q.\]
Let ${\mathscr L}_2(K_Q, H)$ denote the space of all Hilbert-Schmidt operators from $K_Q$ into $H$, then ${\mathscr L}_2(K_Q, H)$ turns out to be a separable Hilbert space under the inner product
\[
\langle L, P\rangle_{{{\mathscr L}_2(K_Q, H)}} =Tr [L QP^*]\hskip 15pt \hbox{for any}\,\,\,\,L,\,\,\,P\in {\mathscr L}_2(K_Q, H).\]
For arbitrarily given
$T\ge 0$, let
$J(t,\omega)$, $t\in[0,T]$, be an ${\mathscr L}_2(K_Q, H)$-valued process. We define the following norm for arbitrary $t\in[0,T]$,
\begin{equation}
\label{11/02/09(10)}
|J|_t :=\biggl\{\mathbb{E}\int^t_0 Tr\Big[J(s,\omega)QJ(s,\omega)^*\Big]ds\biggr\}^{\frac{1}{2}}.
\end{equation}
In particular, we denote all ${\mathscr L}_2(K_Q, H)$-valued measurable processes $J$, adapted to the filtration $\{{\mathscr F}_t\}_{t\le T}$, satisfying $|J|_T <
\infty$ by ${\cal U}^2\big([0,T]; \,{\mathscr L}_2(K_Q, H) \big)$.
The stochastic integral $\int^t_0 J(s, \omega)dW(s) \in H$, $t\ge 0$, may be defined
 for all $J\in {\cal U}^2([0,T];\, {\mathscr L}_2(K_Q, H))$ by
\[
\int^t_0 J(s, \omega)dW(s) = L^2 - \lim_{n\rightarrow \infty}
\sum^n_{i=1} \int^t_0 \sqrt{\lambda_i}J(s, \omega)e_i dB^i_s,\hskip 15pt
t\in [0, T],
\]
where $W(t) = \sum^\infty_{i=1}\sqrt{\lambda_i}B^i_t e_i$. Here $(\lambda_i\ge 0,\,i\in \mathbb{N})$ are the eigenvalues of $Q$ with the corresponding eigenvectors $(e_i,\, i\in \mathbb{N})$, and $(B^i_t,\,i\in \mathbb{N})$ are independent standard real-valued Brownian motions.

In this work, we shall consider the following stochastic functional differential equation  on the Hilbert space $H$,
\begin{equation}
\label{22/03/07(1)}
\begin{cases}
dy(t) =Ay(t)dt +Fy_tdt + BdW(t)\,\,\,\,\,\,\,\,\,\hbox{for any}\,\,\,\,\,\,\,\,t\in [0, T],\\
y(0)=\phi_0\in H,\,\, y_0(\theta)=y(\theta) = \phi_1(\theta)\in H,\,\hskip 10pt \theta\in [-r, 0),
\end{cases}
\end{equation}
for  arbitrarily given  initial datum $\Phi= (\phi_0, \phi_1)\in H\times L^2([-r, 0]; H)$ where $r> 0$ is a given constant and $y_t(\theta):= y(t+\theta)$ for $\theta\in [-r, 0]$, $t\ge 0$. Here $A$ is the infinitesimal generator of a $C_0$-semigroup $e^{tA}$, $t\ge 0$, $B\in {\mathscr L}_2(K_Q, H)$, and $F:\, L^2([-r, 0]; H)\to H$ is some  linear, probably unbounded, operator to be specified later on.

 If $F=0$,  the solution of (\ref{22/03/07(1)}) is called an {Ornstein-Uhlenbeck process\/} which is Gaussian and Markovian. There exists extensive literature on various topics such as Feller semigroups, invariant measures and so on for this process. The reader is referred to, e.g., \cite{acm87}, \cite{dpjz92} and reference cited therein for a comprehensive theory and related topics.  If $F\not=0$, the solution of  (\ref{22/03/07(1)}) is called the so-called {\it retarded\/}  Ornstein-Uhlenbeck process in Hilbert space $H$. To my knowledge, there is little work devoted to the process of (\ref{22/03/07(1)})  in the existing literature, e.g., \cite{acm78}, \cite{kl08(1)}, \cite{kl08(2)},  \cite{kl09(2)} and  \cite{kl09(4)}  which dealt with stationary solutions of the system and related topics.

 Historically, regularity problem for infinite dimensional systems is quite important and it has been investigated by many researchers, e.g., in \cite{dpgkjz87}, \cite{dpjz92} for stochastic evolution equations without memory and in \cite{gdbkkes84}, \cite{ht92} for deterministic functional differential equations among others.   In this work, we are interested in the regularity property of the solution processes of (\ref{22/03/07(1)}). Basically, in the research of the system (\ref{22/03/07(1)}), one of the most important approaches  is to lift the system under investigation  to some expanded space, e.g., $H\times L^2([-r, 0]; H)$ or $C([-r, 0]; H)$, so as that one can consider  a lifted stochastic system without memory, rather than (\ref{22/03/07(1)}) itself. In spite of its obvious advantages, this method introduces, however, significant mathematical difficulties in dealing with regularity problem. For instance, suppose that the operator $A$ in  (\ref{22/03/07(1)})
generates an analytic semigroup, a condition  which is frequently assumed in the investigation of regularity problem, the lifted generator of the system does not generate an analytic semigroup any more on the expanded spaces as above (cf.  \cite{kl09(2)}). In this work, we shall employ a straightforward method to deal with regularity problem by developing a theory of retarded type of Green operators for the system (\ref{22/03/07(1)})
 (cf.  \cite{kl08(1)}, \cite{kl09(1)}). 

The organization of this paper is as follows. We shall introduce in Section 2 a class of fundamental solutions or retarded Green operators for the system (\ref{22/03/07(1)}) and meanwhile review useful notations, definitions and properties to be used in the work. In terms of fundamental solutions, we shall define in Section 3 the so-called retarded stochastic convolutions which naturally arise in the variation of constants formula for solutions of (\ref{22/03/07(1)}). By using a factorization method introduced in \cite{dpgkjz87}, we shall  establish sufficient conditions under which there exists a continuous modification of retarded stochastic convolutions.
By analogy with those in the classical semigroup theory, we shall establish in Section 4 the powerful Yosida approximations for the corresponding deterministic system, i.e., $B=0$, of  (\ref{22/03/07(1)}). Subsequently, Yosida approximations are applied in Section 5 to the investigation of regularity property for a class of retarded linear stochastic functional differential equations.  In Section 6, we proceed to establish a version of Burkholder-Davis-Gundy's inequality for retarded stochastic convolutions whose non-delays counterpart has been known for a while in many a reference, e.g.,  \cite{dpjz92}. Last, we add an Appendix to show some regularity results  for deterministic systems with time delays which play an important role in Section 5.

\section{Fundamental Solutions}

Let $r>0$  and  we denote by $L^2_r=L^2([-r, 0]; H)$ the space of all $H$-valued equivalence classes of measurable functions $\varphi(\theta)$, $\theta\in [-r, 0]$, such that $\int^0_{-r}\|\varphi(\theta)\|^2_Hd\theta<\infty$.  We also denote by $W^{1, 2}([-r, 0]; H)$ the Sobolev space of all $H$-valued function $y$ on $[-r, 0]$ such that $y$ and its distributional derivative belong to $L^2([-r, 0]; H)$. 
Let ${\cal H}$
denote the product Hilbert space $H\times L^2_r$ with its norm and inner product defined, respectively, by
\[
\|\Phi\|_{\cal H}= (\|\phi_0\|_H^2+\|\phi_1\|^2_{L^2_r})^{1/2},\hskip 20pt 
\langle\Phi, \Psi\rangle_{\cal H}= \langle \phi_0, \psi_0\rangle_H+\langle \phi_1, \psi_1\rangle_{L^2_r}\]for all $\Phi=(\phi_0, \phi_1)$, $\Psi=(\psi_0, \psi_1)\in {\cal H}$.

Let  $A:\, {\mathscr D}(A)\subseteq H\to H$ be the infinitesimal generator of a $C_0$-semigroup $e^{tA}$, $t\ge 0$, on $H$ where ${\mathscr D}(A)$ denotes the domain of operator $A$.  Let $T\ge 0$ and assume that  $F:\,W^{1, 2}([-r, 0]; H)\to H$ is a  bounded linear operator such that the map $F$ permits a bounded linear extension $F:\, L^2([-r, T]; H)\to L^2([0, T]; H)$ which is defined by $(Fy)(t)=Fy_t$, $y\in L^2([-r, T]; H)$, with $y_t(\theta) := y(t+\theta)$ for $\theta\in [-r, 0]$, $t\ge 0$. That is, there exists a real number $M_2 >0$ such that 
\begin{equation}
\label{10/02/08(1)}
\int^T_0 \|(Fy)(t)\|^2_Hdt \le M_2 \int^T_{-r} \|y(t)\|^2_Hdt\hskip 15pt \hbox{for any}\hskip 15pt y\in L^2([-r, T]; H).
\end{equation} 
 
Consider the following deterministic functional differential equation on $H$,
\begin{equation}
\label{22/03/07(1234)}
\begin{cases}
dy(t) =Ay(t)dt +Fy_t dt\,\,\,\,\,\,\,\,\,\hbox{for any}\,\,\,\,\,\,\,\,t> 0,\,\\
y(0)=\phi_0,\,\, y_0=\phi_1,\,\hskip 10pt \Phi=(\phi_0, \phi_1)\in {\cal H},
\end{cases}
\end{equation}
and its corresponding functional integral equation
\begin{equation}
\label{17/05/06(3)}
\begin{cases}
y(t) = e^{tA}\phi_0 + \int^t_0 e^{(t-s)A}Fy_sds,\,\,\,\,\,\,t>0,\\
y(0) =\phi_0,\hskip 5pt y_0=\phi_1,\,\,\,\,\,\Phi=(\phi_0, \phi_1)\in  {\cal H}.
\end{cases}
\end{equation} 
It may be shown  that for any $\Phi\in {\cal H}$, the equation  (\ref{17/05/06(3)}) has a unique solution $y(t, \Phi)$ which is called the {\it mild solution\/} of (\ref{22/03/07(1234)}).  For any $x\in H$, we define the  {\it (retarded) fundamental solution\/} or {\it (retarded) Green operator\/} $G(t): (-\infty, \infty)\to {\mathscr L}(H)$ of  ({\ref{17/05/06(3)}) by
\begin{equation}
\label{25/05/06(1)}
G(t)x = \begin{cases}
y(t, \Phi),\hskip 15pt &t\ge 0,\\
0,\hskip 15pt &t< 0,
\end{cases}
\end{equation} 
where $\Phi= (x, 0)$, $x\in H$.  It turns out (cf. \cite{kl08(1)})
 that $G(t)$, $t\ge 0$, is a strongly continuous one-parameter family of bounded linear operators on $H$ such that
\begin{equation}
\label{10/02/09(4)}
\|G(t)\|\le c\cdot e^{\gamma t}, \,\,\,\,\,\,\, t\ge 0,
\end{equation}
 for some constants $c >0$ and $\gamma\in {\mathbb R}^1:=(-\infty, \infty)$. On the other hand, it is easy to see that  $G(t)$ is the unique solution of the functional operator integral equation 
\begin{equation}
\label{25/05/06(2)}
G(t) = \begin{cases}
e^{tA} + \displaystyle\int^t_0 e^{(t-s)A}FG(s+\cdot)ds,\hskip 15pt &t\ge 0,\\
{\rm O},\hskip 15pt &t< 0,
\end{cases}
\end{equation} 
where ${\rm O}$ denotes the null operator on $H$. 
\begin{remark}\rm 
It is worth mentioning that for a particular delay operator $F$ defined by 
\begin{equation}
\label{19/11/07(163)}
{F}\varphi =\sum^m_{i=1}A_i\varphi (-r_i) + \int^0_{-r} A_0(\theta)\varphi(\theta)d\theta,\hskip 20pt \forall\, \varphi\in W^{1, 2}([-r, 0]; H).
\end{equation}
where $0\le r_1\le  \cdots \le r_m\le r$, $A_i\in {\mathscr L}(H)$, $i=1,\cdots, m$, and $A_0(\cdot)\in L^2([-r, 0]; {\mathscr L}(H))$, a similar concept of fundamental solutions was introduced in \cite{sn86}.
\end{remark}

For each function $\varphi:\, [-r, 0]\to H$, we define its right extension function  $\vec{\varphi}$  by
\begin{equation}
\label{12/01/08(1)}
\vec{\varphi}: \,\, [-r, \infty)\to H,\,\,\,\, \vec{\varphi}(t) =
\begin{cases}
\varphi(t),&\hskip 15pt -r\le t\le 0,\\
0,&\hskip 15pt 0<t<\infty.
\end{cases}
\end{equation}
By virtue of (\ref{12/01/08(1)}), it may be shown (cf. \cite{kl08(1)}) that the mild solution of (\ref{22/03/07(1234)}) is represented explicitly by the variation of constants formula
\begin{equation}
\label{25/05/06(4)}
y(t) = G(t)\phi_0 + \int^t_{0} G(t-s)F(\vec{\phi_1})_sds,\hskip 15pt t\ge 0,
\end{equation}
and $y(t) =\phi_1(t)$, $t\in [-r, 0)$. It is useful to introduce the so-called structure operator $S$ defined on the space $L^2([-r, 0]; H)$ by 
\begin{equation}
\label{28/06/06(101)}
\begin{split}
(S\varphi)(\theta) =F\vec{\varphi}_{-\theta},\hskip 10pt \theta\in [-r, 0],\hskip 15pt \forall\, \varphi(\cdot)\in W^{1,2}([-r, 0]; H).
\end{split}
\end{equation}
It is not difficult to show that $S$ can be extended to a  linear and bounded operator from $L^2([-r, 0]; H)$ into itself.  Moreover,  the variation of constants formula for the mild solution of (\ref{22/03/07(1234)}) may be rewritten as
\begin{equation}
\label{29/01/08(1)}
\begin{cases}
y(t) = G(t)\phi_0 + \displaystyle\int^{0}_{-r} G(t+\theta)(S\phi_1)(\theta)d\theta,\hskip 15pt t\ge 0,\\
y(t) =\phi_1(t),\,\,\,\,\,t\in [-r, 0).
\end{cases}
\end{equation}
 In general, the family $G(t)$, $t\in {\mathbb R}^1$, would no longer be a semigroup on $H$. However, we may show  that it is a ``quasi-semigroup" in the sense that
\begin{equation}
\label{27/06/06(30)}
\begin{split}
G(t+s)x =G(t)G(s)x + \int^0_{-r} G(t+\theta)[SG(s+\cdot)x](\theta)d\theta\hskip 15pt \hbox{for all}\hskip 10pt s,\,\,t\ge 0,\hskip 10pt x\in H.
\end{split}
\end{equation}
In association with the operator $S$, we may define a new operator $\tilde S$ on $L^2([-r, 0]; {\mathscr L}(H))$ by
 \begin{equation}
\label{27/06/06(302)}
[\tilde SJ](\theta)x = [(\tilde SJ)x](\theta) := [S(Jx)](\theta),\hskip 10pt x\in H,\hskip 10pt \theta\in [-r, 0],\hskip 10pt a.e.
\end{equation}
 for any $J(\cdot)\in L^2([-r, 0]; {\mathscr L}(H))$.
It is shown that such an operator, still denoted by $S$,  is a linear bounded operator from  $L^2([-r, 0]; {\mathscr L}(H))$ into itself. Indeed, since the operator $S$ in (\ref{28/06/06(101)}) is bounded on $L^2([-r, 0]; H)$, it follows that for some constant $C>0$, there is
 \[
\begin{split}
\int^0_{-r} \|[SJ](\theta)\|^2d\theta & = \int^0_{-r} \sup_{\|x\|_H\le 1}\|[SJ](\theta)x\|_H^2 d\theta= \int^0_{-r} \sup_{\|x\|_H\le 1}\|[S(Jx)](\theta)\|_H^2d\theta\\
&\le C\int^0_{-r}  \sup_{\|x\|_H\le 1}\|(Jx)(\theta)\|^2_Hd\theta = C \int^0_{-r}\sup_{\|x\|_H\le 1}\|J(\theta)x\|^2_Hd\theta\\
&\le C \int^0_{-r}\|J(\theta)\|^2d\theta.
\end{split}
\]
This implies that $S$ is a linear bounded operator on  $L^2([-r, 0]; {\mathscr L}(H))$. Moreover, on this occasion the relation (\ref{27/06/06(30)}) yields that  
\begin{equation}
\label{27/06/06(300)}
G(t+s) =G(t)G(s) + \int^0_{-r} G(t+\theta)[SG(s+\cdot)](\theta)d\theta\hskip 20pt \hbox{for all}\hskip 10pt s,\,\,t\ge 0.
\end{equation}

\section{Continuous Sample Paths}

Let $L^2_{{\mathscr F}_0}(\Omega, {\mathscr F}, {\mathbb P}; {\cal H})$ denote the space of all ${\cal H}$-valued mappings  $\Psi(\omega)= (\psi_0(\omega), \psi_1(\cdot, \omega))$ defined on $(\Omega, {\mathscr F}, {\mathbb P})$ such that both $\psi_0$ and $\psi_1(\theta)$ are ${\mathscr F}_0$-measurable for any $\theta\in [-r, 0]$ and satisfy
\[
{\mathbb E}\|\Psi\|^2_{\cal H} ={\mathbb E}\|\psi_0\|^2_H + {\mathbb E}\|\psi_1\|^2_{L^2([-r, 0]; H)}<\infty.\]
We shall be concerned about  the following stochastic functional evolution equation  on the Hilbert space $H$,
\begin{equation}
\label{10/08/07(1)}
\begin{cases}
dy(t) =[Ay(t) +Fy_t]dt + BdW(t)\,\,\,\,\,\,\,\,\,\hbox{for any}\,\,\,\,\,\,\,\,t\in [0, T],\,\\
y(0)=\psi_0,\,\, y_0=\psi_1,\,\hskip 10pt \Psi=(\psi_0, \psi_1)\in L^2_{{\mathscr F}_0}(\Omega, {\mathscr F}, {\mathbb P}; {\cal H}),
\end{cases}
\end{equation}
where $B\in  {\mathscr L}_2(K_Q, H)$, $W(t)$ is a $K$-valued $Q$-Wiener process on $(\Omega, {\mathscr F}, {\mathbb P})$ and the delay operator $F$ is given as in Section 2.

For any $t\ge 0$, let $Q_t= \int^t_0 G(s)BQB^*G^*(s)ds$ where $G^*(s)$ denotes the adjoint operator of $G(s)$ for any $s\ge 0$. For the problem  (\ref{10/08/07(1)}), it was shown in \cite{kl08(1)}  that for any $\Psi=(\psi_0, \psi_1)\in L^2_{{\mathscr F}_0}(\Omega, {\mathscr F}, {\mathbb P}; {\cal H})$, if  
\begin{equation}
\label{24/08/07(1)}
\hbox{Tr}\,[Q_t]= \int^t_0 \hbox{Tr}\,[G(s)BQB^*G^*(s)]ds<\infty\,\,\,\,\,\hbox{for any}\,\,\,\,\,\, t\in [0, T],
\end{equation}
then
there exists a unique mild solution $y(t, \Psi)$ of  (\ref{10/08/07(1)}). Moreover,  this solution is mean square continuous with sample paths (almost surely) in $L^2([0, T]; H)$ for each $T\ge 0$ and it may be explicitly represented in terms of fundamental solutions $G(t)$, $t\in {\mathbb R}^1$, by
\begin{equation}
\label{22/03/07(5)}
\begin{split}
y(t, \Psi) = G(t)\psi_0 + \int^0_{-r} G(t+\theta)S\psi_1(\theta)d\theta + \int^t_{0} G(t-s)BdW(s),\hskip 10pt t\in [0, T],
\end{split}
\end{equation} 
where $S$ is the structure operator defined in (\ref{28/06/06(101)}). 

The aim of this section is to show that under a slightly stronger version of (\ref{24/08/07(1)}), the solution $(y(t, \Psi),\,t\ge 0)$, or equivalently, the retarded stochastic convolution $W_G^B(t) := \int^t_{0} G(t-s)BdW(s)$, $t\ge 0$, has a version with continuous sample paths. To this end, we first establish a useful lemma.
\begin{lemma}
\label{16/04/09(1)}
Let $T\ge 0$ and $\alpha\in (0, 1/2)$. Assume that function $z(t, s): [0, T]\times [0, T]\to H$ is continuous in $t$ and for any $t\in [0, T]$, $z(t, \cdot)\in L^{m}([0, T]; H)$ for some natural number $m>1/\alpha$, then the function 
\[
l(t) =\int^t_0 (t-s)^{\alpha-1}z(t, s)ds,\hskip 15pt t\in [0, T],\]
is continuous on $[0, T]$.
\end{lemma}
\begin{proof}
First note that if $z(t, s)$ is an $H$-valued continuous function on $[0, T]\times [0, T]$, then the function $l(t)$ is continuous on $[0, T]$. 

Now suppose that $z(t, s):\, [0, T]\times [0, T]\to H$ is continuous in $t$ and for any $t\in [0, T]$, $z(t, \cdot)\in L^{m}([0, T]; H)$ where $m>1/\alpha$,  then it is easy to see that there exist a family of continuous functions $z_n(t, s)$ on $[0, T]\times [0, T]$ such that 
\begin{equation}
\label{16/05/09(1)}
\sup_{t\in [0, T]}\int^t_0 \|z(t, s)-z_n(t, s)\|^{m}_H ds\to 0\hskip 10pt \hbox{as}\hskip 10pt n\to\infty.
\end{equation}
On the other hand, we have by virtue of  H\"older inequality that
\begin{equation}
\begin{split}
\label{19/04/09(1)}
\|l(t)\|^{m}_H &\le \left(\int^t_0 (t-s)^{\frac{(\alpha-1)m}{m-1}}ds\right)^{\frac{m-1}{m}}\int^t_0 \|z(t, s)\|^{m}_Hds\\
& \le C_{\alpha, m, T}\int^t_0 \|z(t, s)\|^m_Hds,\hskip 15pt t\in [0, T],
\end{split}
\end{equation}
where \[
C_{\alpha, m, T}= \left(\int^T_0 s^{\frac{(\alpha-1)m}{m-1}}ds\right)^{\frac{m-1}{m}} = \left(\frac{m-1}{\alpha m-1}\right)^{\frac{m-1}{m}}T^{\frac{\alpha m-1}{m}} >0.\]
 This immediately yields that 
\begin{equation}
\label{19/04/09(2)}
\sup_{t\in [0, T]}\|l(t)\|^{m}_H \le C_{\alpha, m, T}\sup_{t\in [0, T]}\int^t_0 \|z(t, s)\|^{m}_Hds.
\end{equation}
Therefore, for the  function $z(t, s):\, [0, T]\times [0, T]\to H$ there exist, in view of (\ref{16/05/09(1)}) and (\ref{19/04/09(2)}),
  a family of continuous functions $z_n(t, s)$ on $[0, T]\times [0, T]$ such that
\begin{equation}
\label{17/04/09(1)}
\sup_{t\in [0, T]}\|l(t)-l_n(t)\|^{m}_H \le C_{\alpha, m, T}\sup_{t\in [0, T]}\int^t_0 \|z(t, s)-z_n(t, s)\|^{m}_Hds\to 0\hskip 10pt \hbox{as}\hskip 10pt n\to\infty,
\end{equation}
where
\[ 
l_n(t) =\int^t_0 (t-s)^{\alpha-1}z_n(t, s)ds,\hskip 15pt t\in [0, T],\hskip 15pt n\in {\mathbb N}.\]
Since $l_n(t)$ is continuous on $[0, T]$, (\ref{17/04/09(1)}) implies the continuity of $l(t)$ on $[0, T]$. The proof is thus complete. 
\end{proof}

\begin{theorem}
\label{08/04/10(1)}
Assume that for some $\alpha>0$ and $T\ge 0$, the relation
\begin{equation}
\label{27/03/09(10)}
 \int^T_0 t^{-2\alpha}Tr[G(t)BQB^*G(t)^*]dt <\infty
\end{equation}
holds. Then the retarded stochastic convolution  $W^B_G(t)=\int^t_0 G(t-s)BdW(s)$ has a continuous modification on $[0, T]$.
\end{theorem}
\begin{proof} Without loss of generality,  fix a number $\alpha\in (0, 1/2)$ and note the following elementary identity
\begin{equation}
\label{27/03/09(1)}
\int^t_u (t-s)^{\alpha-1}(s-u)^{-\alpha}ds = \frac{\pi}{\sin \pi\alpha}\hskip 15pt \hbox{for any}\hskip 15pt u\le s\le t\le T.
\end{equation}
By virtue of (\ref{27/03/09(1)}), it is easy to see that
\begin{equation}
\label{27/03/09(2)}
W_G^B(t) = \frac{\sin \pi\alpha}{\pi}\int^t_0 G(t-u)\Big[\int^t_u (t-s)^{\alpha-1}(s-u)^{-\alpha}ds\Big]BdW(u).
\end{equation}
In view of the well-known stochastic Fubini theorem and quasi-semigroup property (\ref{27/06/06(300)}) of $G(t)$, one can rewrite (\ref{27/03/09(2)}) for any $t\in [0, T]$ as
 \begin{equation}
\label{20/04/09(1)}
\begin{split}
&W_G^B(t)\cr 
&=  \frac{\sin \pi\alpha}{\pi}\int^t_0 (t-s)^{\alpha-1}\int^s_0 G(t-s+s-u)(s-u)^{-\alpha}BdW(u)ds\cr
&=  \frac{\sin \pi\alpha}{\pi}\int^t_0 (t-s)^{\alpha-1}\int^s_0 \Big[\int^0_{-r} G(t-s+\theta)[SG(s-u+\cdot)](\theta)d\theta +
G(t-s)G(s-u)\Big]\cr
&\hskip 150pt \cdot (s-u)^{-\alpha}BdW(u)ds\cr
&= \frac{\sin \pi\alpha}{\pi}\int^t_0 (t-s)^{\alpha-1}\int^s_0 (s-u)^{-\alpha}\int^0_{-r} G(t-s+\theta)[SG(s-u+\cdot)](\theta)Bd\theta dW(u)ds\cr
&\,\,\,\,\,\,\,\, + \frac{\sin \pi\alpha}{\pi}\int^t_0 (t-s)^{\alpha-1}G(t-s)\int^s_0 G(s-u)(s-u)^{-\alpha}BdW(u)ds\cr
&=: \frac{\sin \pi\alpha}{\pi}(I_1(t) +I_2(t)).
\end{split}
\end{equation}
We first show the existence of a continuous modification for the term $I_1(t)$. To this end, let us rewrite the term $I_1(t)$  as 
\[
I_1(t) = \int^t_0 (t-s)^{\alpha-1}Z(t, s)ds,\hskip 20pt t\in [0, T],\]
where 
\[
Z(t, s) = \int^s_0 (s-u)^{-\alpha}\int^0_{-r}G(t-s+\theta)[SG(s-u+\cdot)](\theta)Bd\theta dW(u),\hskip 15pt s\in [0, t],\]
and its covariance operator is
\begin{equation}
\label{27/03/09(20)}
\begin{split}
\hbox{Cov}\,Z(t, s) &= \int^s_0 (s-u)^{-2\alpha}\int^0_{-r} G(t-s+\theta)[SG(s-u+\cdot)](\theta)Bd\theta\cr
&\,\,\,\,\,\,\,\,\,\cdot Q \Big[\int^0_{-r} G(t-s+\theta)[SG(s-u+\cdot)](\theta)Bd\theta\Big]^*du, \hskip 15pt s\in [0, t].
\end{split}
\end{equation}
Since $\|G(t)\|\le ce^{\gamma t}$, $c>0$, $\gamma\in {\mathbb R}^1$, for all $t\ge 0$ and $S$ is bounded on $L^2([-r, 0]; {\mathscr L}(H))$, the relations (\ref{27/03/09(20)}) and (\ref{27/03/09(10)}), together with H\"older inequality, imply that   for any $s\in [0, t]$, $t\le T$, 
\begin{equation}
\label{27/03/09(21)}
\begin{split}
\hbox{Tr}\,[\hbox{Cov}\,Z(t, s)] &=  \int^s_0 (s-u)^{-2\alpha}Tr\Big[\int^0_{-r} G(t-s+\theta)[SG(s-u+\cdot)](\theta)Bd\theta\cr
&\,\,\,\,\,\,\,\,\,\cdot Q \Big(\int^0_{-r} G(t-s+\theta)[SG(s-u+\cdot)](\theta)Bd\theta\Big)^*\Big]du\cr
&\le  \|S\|^2 r\int^s_0 (s-u)^{-2\alpha} \int^0_{-r} Tr[G(s-u+\theta)BQB^*G(s-u+\theta)^*]d\theta du\cr
&\le   \|S\|^2 r\int^s_0 \int^0_{-r} (u-\theta)^{-2\alpha}Tr[G(u)BQB^*G(u)^*]d\theta du\cr
&\le   \|S\|^2 r^2\int^T_0 u^{-2\alpha}Tr[G(u)BQB^*G(u)^*]du<\infty.
\end{split}
\end{equation}
This shows that for any $t\in [0, T]$, $Z(t, s)$, $s\in [0, t]$,  is not only Gaussian but also has a finite trace covariance operator. Then, by using Corollary 2.17 in \cite{dpjz92} and (\ref{27/03/09(21)}),  one can choose a natural number $m>1/\alpha$ and find a number $C_m>0$ such that for any $t\in [0, T]$,
\begin{equation}
\begin{split}
{\mathbb E}\Big(\int^T_0 \|Z(t, s)\|^{m}_Hds\Big) &= {\mathbb E}\Big(\int^t_0 \|Z(t, s)\|^{m}_Hds\Big)\\ 
&\le C_m\int^T_0 \sup_{t\in [0, T]}\big(\hbox{Tr}\,[\hbox{Cov}\,Z(t, s)]\big)^{m/2} ds\\
&\le C_mT \Big[  \|S\|^2 r^2\int^T_0 u^{-2\alpha}Tr[G(u)BQB^*G(u)^*]du\Big]^{m/2}<\infty,
\end{split}
\end{equation}
 which, together with Lemma \ref{16/04/09(1)} and the strong continuity of $G(t)$, $t\in {\mathbb R}^1$, implies that $I_1(t)= \int^t_0 (t-s)^{\alpha-1}Z(t, s)ds$ has a continuous modification on $[0, T]$. 

In a similar way,  we can show that there exists a continuous modification of $I_2(t)$, $t\in [0, T]$. 
The existence of a continuous modification for $I_1(t)$ and $I_2(t)$ implies further that the stochastic convolution $W^B_G(t)$ has a continuous modification on $[0, T]$.  The proof is thus complete. 
\end{proof}
\begin{corollary}
Let  $\Psi=(\psi_0, \psi_1)\in L^2_{{\mathscr F}_0}(\Omega, {\mathscr F}, {\mathbb P}; {\cal H})$. Assume that the relation  (\ref{27/03/09(10)}) holds for some $\alpha>0$ and $T\ge 0$, then the mild solution of Eq. (\ref{10/08/07(1)}) has a version with continuous sample paths on $[0, T]$.
\end{corollary}
\begin{proof}
The claim is immediate from Theorem \ref{08/04/10(1)}, (\ref{22/03/07(5)}) and the fact that the fundamental solution $G(t)$, $t\ge 0$, is a strongly continuous one-parameter family of bounded linear operators on $H$. 
\end{proof}

\section{Retarded Yosida Approximant}

Suppose that $A$ is the infinitesimal generator of some strongly continuous semigroup $e^{tA}$, $t\ge 0$, on the Hilbert space $H$. Recall that we may define the following  {\it Yosida approximants\/}
\[
A_n := AJ_n= A (nR(n, A)) = n^2 R(n, A) -nI,\]
which are bounded operators for  each $n\in \rho(A)$, the resolvent set of $A$, and commute with one another. Here $J_n = nR(n, A)$ and $R(n, A)$ is the resolvent operator of $A$ for any $n\in \rho(A)$.
 It may be shown that 
\[
e^{ tA}x = \lim_{n\to\infty}e^{tA_n}x\hskip 15pt \hbox{for each}\hskip 10pt x\in H,\]
 and the family of operator $\{e^{tA_n}\}_{n\ge 1}$ is thus called the {\it Yosida approximants\/} of $e^{tA}$, $t\ge 0$.
In this section, we shall consider Yosida approximation for fundamental solution $G(t)$, $t\ge 0$ and meanwhile establish useful properties which will be applied in the next section to the regularity problem for stochastic functional evolution equations. 
  
Let ${\mathscr L}_s(H)$ denote the family of all bounded linear operators on $H$, endowed with the strong operator topology, i.e., the local convex topology generated by the following seminorms 
\[
p_x(B) := \|Bx\|_H,\hskip 10pt B\in {\mathscr L}(H),\hskip 10pt x\in H.\]
Let $C([0, T]; {\mathscr L}_s(H))$ be the space of all strongly continuous functions from $[0, T]$ into ${\mathscr L}_s(H)$. Thus $J(\cdot)\in C([0, T]; {\mathscr L}_s(H))$ if and only if $J(t)\in {\mathscr L}(H)$ for each $t\in [0, T]$ and $t\to J(t)x$ is continuous for each $x\in H$. By virtue of the well-known uniform boundedness principle, it may be shown (cf. \cite{kern00}) that this space is a Banach space under the norm 
\[
\|J\|_{max} := \sup_{t\in [0, T]}\|J(t)\|,\hskip 15pt J\in C([0, T]; {\mathscr L}_s(H)).\] 
On the other hand, note that for any $J(\cdot)\in C([0, T]; {\mathscr L}_s(H))$, we can extend it uniquely to obtain a mapping $\tilde J(\cdot)$ on $[-r, T]$ such that $\tilde J(t)=J(t)$ as $t\in [0, T]$ and $\tilde J(t)=0$ as $t\in [-r, 0)$. We shall always identity $J(\cdot)\in C([0, T]; {\mathscr L}_s(H))$ with such an extension in the sequel when no confusion is possible. 

\begin{definition}\rm
Let $e^{tA}$, $t\ge 0$, be a strongly continuous semigroup on $H$. For any $T\ge 0$, the operator $V$ defined by 
\[
VJ(t)x := \int^t_0 e^{(t-s)A}FJ(s+\cdot)xds,\hskip 15pt t\in [0, T],\hskip 15pt x\in H,\]
on $J(\cdot)\in C([0, T]; {\mathscr L}_s(H))$ is called  the {\it retarded Volterra operator}.
\end{definition}

\begin{lemma}\label{08/12/08(3)}
The retarded Volterra operator $V$ is a bounded linear operator in the space $C([0, T]; {\mathscr L}_s(H))$. Moreover, it satisfies that for any $m\in {\mathbb N}$,
\begin{equation}
\label{02/04/10(1)}
\|V^m\|\le \kappa^m/m!
\end{equation}
where $\kappa = MTM_2^{1/2}>0$, $M= \sup_{t\in [0, T]}\|e^{tA}\|$ and $M_2>0$ is given as in (\ref{10/02/08(1)}). 
\end{lemma}
\begin{proof}
It is clear that $V$ is a linear operator. For the boundedness, we may employ H\"older inequality and (\ref{10/02/08(1)}) to get that for any $J(\cdot)\in C([0, T]; {\mathscr L}_s(H))$ and $t\in [0, T]$,
\begin{equation}
\label{08/12/08(1)}
\begin{split}
\|VJ(t)\| &\le \int^t_0 \sup_{\|x\|_H\le 1} \|e^{(t-s)A}FJ(s+\cdot) x\|_Hds\\
&\le Mt^{1/2}\Big(\int^t_0 \sup_{\|x\|_H\le 1} \|FJ(s+\cdot) x\|^2_Hds\Big)^{1/2}\\
&\le Mt^{1/2}M_2^{1/2}\Big(\int^t_{-r}\sup_{\|x\|_H\le 1} \|J(s)x\|^2_Hds\Big)^{1/2}\\
&\le MtM_2^{1/2}\|J(t)\|,
\end{split}
\end{equation}
which implies that $\|V\|\le \kappa$ where $\kappa= MTM_2^{1/2}>0$. 

The general form (\ref{02/04/10(1)}) can be easily shown by using (\ref{08/12/08(1)}) and implementing  induction on $m\in {\mathbb N}$.
The proof is thus complete.
\end{proof}

\begin{proposition}
\label{14/04/09(1)} 
The fundamental solution $G(t)$, $t\in [0, \infty)$, of the equation (\ref{22/03/07(1234)}) may be explicitly represented as 
\begin{equation}
\label{14/04/09(2)}
G(t) = \sum^\infty_{m=0} G(m, t),\hskip 15pt t\ge 0,
\end{equation}
where $G(0, t) := e^{tA}$, $t\ge 0$, and for each $x\in H$ and $m\ge 1$,
\[
G({m+1}, t)x := VG(m, t)x =\int^t_0 e^{(t-s)A}FG(m, s+\cdot)xds,\,\,\,\,\,t\ge 0.\]
Moreover, the series (\ref{14/04/09(2)}) converges in the operator norm uniformly with respect to $t$ on any compact interval in ${\mathbb R}_+$.
\end{proposition}
\begin{proof}
From Lemma \ref{08/12/08(3)}, it follows that for the Volterra operator $V$, its resolvent $R(\lambda, V)$ exists at $\lambda=1$ and is given by 
\[
R(1, V) = (I-V)^{-1}= \sum^\infty_{m=0}V^m.\]
On the other hand, by virtue of (\ref{25/05/06(2)}) we have for any $x\in H$ that 
\[
G(t)x= R(1, V)e^{tA}x = \sum^\infty_{m=0} V^m e^{tA}x,\,\,\,\,\,t\ge 0.\]
Therefore, let $G(m, t) =V^me^{tA}$, $m\ge 1$, and $G(0, t):= e^{tA}$, $t\ge 0$, then we have that 
\[
G(t)x= \sum^\infty_{m=0} G(m, t)x,\,\,\,\,\, t\ge 0,\,\,\,\,\,x\in H,\]
and
\[
G({m+1}, t)x = VG(m, t)x =\int^t_0 e^{(t-s)A}FG(m, s+\cdot) xds,\,\,\,\,\,m\ge 1,\,\,\,\,\,x\in H.\]
Finally, the uniform convergence of  (\ref{14/04/09(2)}) in the operator norm can be deduced from the fact that $G(m, t)=V^m e^{tA}$ and $\|V^m\|\le \kappa^m/m!$, $m\ge 1$.
The proof is complete now.
\end{proof}

Let $A_n$ be the Yosida approximants of $A$ and consider the following deterministic functional differential  equation for each $n\in {\mathbb N}$,
 \begin{equation}
\label{22/03/07(12344)}
\begin{cases}
dy(t) =A_ny(t)dt +Fy_t dt\,\,\,\,\,\,\,\,\,\hbox{for any}\,\,\,\,\,\,\,\,t> 0,\,\\
y(0)=\phi_0,\,\, y_0=\phi_1,\,\hskip 10pt \Phi=(\phi_0, \phi_1)\in {\cal H}.
\end{cases}
\end{equation}
By analogy with $G(t)$, $t\ge 0$, in Section 2, we can define the corresponding fundamental solutions $G_n(t)$, $n\in {\mathbb N}$, for the equations (\ref{22/03/07(12344)})
which is a family of strongly continuous bounded linear operators on $H$ and satisfies the following equations
\begin{equation}
\label{25/05/06(2034)}
G_n(t) = \begin{cases}
e^{tA_n} + \displaystyle\int^t_0 e^{(t-s)A_n}FG_n(s+\cdot)ds,\hskip 15pt &t\ge 0,\\
{\rm O},\hskip 15pt &t< 0.
\end{cases}
\end{equation} 
Since $\|e^{tA_n}\|\le M e^{\alpha t}$, $t\ge 0$, for some $M\ge 1$,  $\alpha>0$ and  large $n\in {\mathbb N}$ (see, e.g., (A.13) in \cite{dpjz92}),  we can deduce by  (\ref{25/05/06(2034)}) and the well-known Gronwall inequality that
 \begin{equation}
\label{19/05/11(1)}
\|G_n(t)\|\le c e^{\gamma t},\,\,\,\,\, t\ge 0,\,\,\,\,\hbox{for some}\,\,\,\, c>0,\,\,\,\gamma>0\,\,\,\,\hbox{and large}\,\,\,\,n\in {\mathbb N}.
\end{equation}
In a similar way,  we can show that for each fixed $n\in {\mathbb N}$, $G_n(t)$, $t\in [0, \infty)$, is a uniformly norm continuous family in ${\mathscr L}(H)$, i.e., $G_n(t):\, [0, \infty)\to {\mathscr L}(H)$, by using (\ref{25/05/06(2034)}) and the fact that $e^{tA_n}: [0, \infty)\to {\mathscr L}(H)$ is uniformly norm continuous.
\begin{proposition}
\label{09/05/09(1213)}
Let $A_n$ be the Yosida approximants of $A$ and $G_n(t)$, $n\in {\mathbb N}$, be the corresponding fundamental solutions of the  equations (\ref{22/03/07(12344)}), then  for any $x\in H$,
\[
G_n(t)x \to G(t)x\hskip 15pt \hbox{as}\hskip 15pt n\to\infty.\]
Moreover, this limit converges uniformly with respect to $t$ on any compact interval in ${\mathbb R}_+$.
\end{proposition}
\begin{proof}
By virtue of Proposition \ref{14/04/09(1)}, the fundamental solutions $G(t)$ and $G_n(t)$, $t\in [0, \infty)$, of the equations  (\ref{22/03/07(1234)}) and (\ref{22/03/07(12344)}) may be represented, respectively, by
\begin{equation}
\label{20/04/09(10)}
G(t) = \sum^\infty_{m=0} G(m, t),\hskip 20pt G_n(t) = \sum^\infty_{m=0} G_n(m, t),\hskip 10pt t\in {\mathbb R}_+,
\end{equation}
where 
\[
G(m, t) =V^m e^{tA},\,\,\,m\ge 1;\,\,\,\,\,\,\,G(0, t)=e^{tA},\,\,t\ge 0,\]
and
\[
G_n(m, t) =V^m e^{tA_n},\,\,\,m\ge 1;\,\,\,\,\,\,\,\,G_n(0, t)=e^{tA_n},\,\,t\ge 0.\]
Since $G_n(m, t)x\to G(m, t)x$ as $n\to\infty$ for each $m\in {\mathbb N}$ and $x\in H$, and the series in (\ref{20/04/09(10)}) converges in the operator norm topology uniformly with respect to $t$ on any compact interval in ${\mathbb R}_+$, it follows that $G_n(t)\to G(t)$  converges in the strong sense as $n\to\infty$ uniformly with respect to $t$ on any compact interval in ${\mathbb R}_+$. The proof is thus complete.
\end{proof}

\section{Regularity Property}

In this section, we shall concern about regularity property for a class of linear functional stochastic evolution equations. To this end, we shall focus throughout this section on a specific  delay operator $F$ (cf. Remark 2.1 or \cite{kl08(1)}) of (\ref{19/11/07(163)})
which is given by 
\begin{equation}
\label{09/04/10(1)}
F\varphi   =B_1\varphi (-r) + \int^0_{-r} a(\theta)B_0\varphi(\theta)d\theta,\,\,\,\,\,\varphi\in C([-r, 0]; H),
\end{equation}
where $B_1,\,B_0\in {\mathscr L}(H)$ and $a(\cdot)\in L^1([-r, 0]; {\mathbb R}^1)$. For any $B\in {\mathscr L}_2(K_Q, H)$, we intend to consider the regularity of the following stochastic functional differential equation on $H$
\begin{equation}
\label{19/11/07(11)}
 \begin{cases}
{dy(t)} = Ay(t)dt + B_1y(t-r)dt + \displaystyle\int^0_{-r} a(\theta)B_0 y(t+\theta)d\theta dt +BdW(t),\hskip 15pt t>0,\\
y(0) = \phi_0,\hskip 10pt y_0=\phi_1,\hskip 10pt \Phi=(\phi_0, \phi_1)\in {\cal H}.
\end{cases}
\end{equation}

\begin{lemma} 
\label{12/04/2010(2)}
Suppose that $1\le p<\infty$ and $a(\cdot)\in L^q([-r, 0]; {\mathbb R}^1)$, $1/p+1/q=1$.  Then for the delay operator $F$ defined in (\ref{09/04/10(1)}), it  permits a bounded linear extension $F:\, L^{p}([-r, T]; H)\to L^{p}([0, T]; H)$ which is defined by $(Fy)(t)=Fy_t$, $y\in L^{p}([-r, T]; H)$. That is,  there exists a real number $M_p >0$ such that 
\begin{equation}
\label{10/02/08(102)}
\int^T_0 \|(Fy)(t)\|^{p}_Hdt \le M_p \int^T_{-r} \|y(t)\|^{p}_Hdt\hskip 15pt \hbox{for any}\hskip 15pt y\in L^{p}([-r, T]; H)
\end{equation} 
where 
\[
M_p =\Big\{\|B_1\| +\|B_0\|\|a(\cdot)\|_{{L^q([-r, 0]; {\mathbb R}^1)}}\cdot r^{1/p}\Big\}^p>0.\]
\end{lemma}
\begin{proof} Since $W^{1, p}([-r, 0]; H)$ is continuously embedded in $C([-r, 0]; H)$, we may note that $F:\, W^{1, p}([-r, 0]; H) \to H$ is linear and bounded. On the other hand, for fixed $T\ge 0$ and any $y\in W^{1, p}([-r, T]; H)$, $1\le p<\infty$, one has by using H\"older inequality and Fubini's theorem that
\begin{equation}
\label{19/11/07(14)}
\begin{split}
&\Big(\int^{t}_0 \|Fy_s\|^p_Hds\Big)^{1/p}\\
&\le \Big(\int^{t}_0 \|B_1\|^p\|y(s-r)\|^p_H ds\Big)^{1/p} + \|B_0\|\|a(\cdot)\|_{{L^q([-r, 0]; {\mathbb R}^1)}} \Big(\int^{t}_0 \int^0_{-r} \|y(s+\theta)\|_H^p d\theta ds\Big)^{1/p}\\
&= \|B_1\|\Big(\int^{t}_0 \|y(s-r)\|^p_H ds\Big)^{1/p} + \|B_0\|\|a(\cdot)\|_{{L^q([-r, 0]; {\mathbb R}^1)}}\\
&\hskip 10pt \cdot\Big(\int^0_{-r}(s+r)\|y(s)\|^p_Hds + \int^{t-r}_0 r \|y(s)\|_H^p ds + \int^t_{t-r}(t-s)\|y(s)\|^p_Hds\Big)^{1/p}\\
&\le \Big[ \|B_1\| + \|B_0\|\|a(\cdot)\|_{L^q([-r, 0]; {\mathbb R}^1)}\cdot r^{1/p}\Big] \Big(\int^{t}_{-r} \|y(s)\|_H^p ds\Big)^{1/p}
\end{split}
\end{equation}
where $1/p+1/q=1$. As $W^{1, p}([-r, T]; H)$ is dense in $L^p([-r, T]; H)$, the operator ${F}$ can be extended to $L^p([-r, T]; H)$ so that (\ref{19/11/07(14)}) remains valid for all $y\in L^p([-r, T]; H)$
and  the constant $M_p>0$  in (\ref{10/02/08(102)}) is clearly given by 
\[
M_p =\Big\{\|B_1\| +\|B_0\|\|a(\cdot)\|_{{L^q([-r, 0]; {\mathbb R}^1)}}\cdot r^{1/p}\Big\}^p>0.\] 
The proof is complete now.
\end{proof}

\begin{remark}\label{12.04/10(3)}
\rm 
We can extend the structure operator $S$ introduced in (\ref{28/06/06(101)}) (resp. $S$ in 
(\ref{27/06/06(302)})) to the space $L^{p}([-r, 0]; H)$ (resp. $L^{p}([-r, 0]; {\mathscr L}(H))$), $1\le p<\infty$, by defining  
\begin{equation}
\label{28/06/06(1010)}
\begin{split}
(S\varphi)(\theta) =F\vec{\varphi}_{-\theta},\hskip 10pt \theta\in [-r, 0],\hskip 10pt \hbox{almost everywhere}\hskip 10pt\forall\, \varphi(\cdot)\in L^{p}([-r, 0]; H).
\end{split}
\end{equation}
It is shown that the operator $S$ can be extended to  a  linear and bounded operator from $L^{p}([-r, 0]; H)$ into itself. Indeed, we have 
by virtue of (\ref{10/02/08(102)}) that 
\begin{equation}
\label{12/04/10(3)}
\begin{split}
\int^0_{-r} \|S\varphi(\theta)\|^{p}_Hd\theta &= \int^r_0 \|F\vec{\varphi}_{-\theta}\|^{p}_Hd\theta\le M_p \int^{r}_{-r}\|\vec{\varphi}(\theta)\|^{p}_Hd\theta\\
&=M_p \int^0_{-r} \|\varphi(\theta)\|^{p}_Hd\theta,\hskip 15pt \forall\, \varphi(\cdot)\in W^{1, p}([-r, 0]; H).
\end{split}
\end{equation}
\end{remark}

\begin{proposition}
\label{10/05/09(2)} Let $G_n(t)$ be the retarded Yosida approximants of $G(t)$, $B\in {\mathscr L}_2(K_Q, H)$ and $F$ be given as in (\ref{09/04/10(1)}). Then for any $T\ge 0$, the stochastic convolutions $W^B_G(t)$ and $W^B_{G_n}(t)$, $t\in [0, T]$, satisfy that 
\begin{equation}
\label{20/04/09(25)}
\lim_{n\to\infty}{\mathbb E}\sup_{t\in [0, T]}\|W^B_G(t) - W^B_{G_n}(t)\|^p_H=0,\hskip 20pt p>2.
\end{equation}
\end{proposition}
\begin{proof}
Let $\alpha \in (1/p, 1/2)$ and recall that for any $t\in [0, T]$, we have from (\ref{20/04/09(1)})
that
 \begin{equation}
\label{01/03/09(2002)}
\begin{split}
W_G^B(t) 
=  \frac{\sin \pi\alpha}{\pi}\Big\{\int^t_0 (t-s)^{\alpha-1}Y(t, s)ds +\int^t_0 (t-s)^{\alpha-1}G(t-s)Z(s)ds\Big\}
\end{split}
\end{equation}
where 
\[
Y(t, s) = \int^s_0 (s-u)^{-\alpha}\int^0_{-r} G(t-s+\theta)[SG(s-u+\cdot)](\theta)Bd\theta dW(u),\,\,\,\,\,s\in [0, T],\]
and
\[
Z(s) = \int^s_0 G(s-u)(s-u)^{-\alpha}BdW(u),\,\,\,\,\,s\in [0, T].\]
In a similar way, we have for $G_n$ that 
 \begin{equation}
\label{01/03/09(2003)}
\begin{split}
W_{G_n}^B(t) 
=  \frac{\sin \pi\alpha}{\pi}\Big\{\int^t_0 (t-s)^{\alpha-1}Y_n(t, s)ds +\int^t_0 (t-s)^{\alpha-1}G_n(t-s)Z_n(s)ds\Big\}.
\end{split}
\end{equation}
where 
\[
Y_n(t, s) = \int^s_0 (s-u)^{-\alpha}\int^0_{-r} G_n(t-s+\theta)[SG_n(s-u +\cdot)](\theta)Bd\theta dW(u),\,\,\,\,\,s\in [0, T],\]
and
\[
Z_n(s) = \int^s_0 G_n(s-u)(s-u)^{-\alpha}BdW(u),\,\,\,\,\,s\in [0, T].\]
Hence, we can write for any $t\in [0, T]$ that
 \begin{equation}
\label{20/04/09(20)}
\begin{split}
W^B_G(t) - W^B_{G_n}(t) &= \frac{\sin \pi\alpha}{\pi}\Big\{ \int^t_0 (t-s)^{\alpha-1}[Y(t, s)-Y_n(t, s)]ds\\
&\,\,\,\,\,\,\, + \int^t_0 (t-s)^{\alpha-1}\Big[G(t-s)Z(s) -G_{n}(t-s)Z_n(s)\Big]ds\Big\}\\
&=: \frac{\sin \pi\alpha}{\pi}(J_1(n, t) +J_2(n, t)).
\end{split}
\end{equation}
We first show  that 
\[
\lim_{n\to\infty}{\mathbb E}\sup_{t\in [0, T]}\|J_1(n, t)\|^p_H=0.\]
Indeed, by virtue of H\"older inequality and $\alpha\in (1/p, 1/2)$, we have that
  \begin{equation}
\label{20/04/09(21)}
\begin{split}
{\mathbb E}\sup_{t\in [0, T]}\|J_1(n, t)\|^p_H &\le \left(\int^T_0 s^{(\alpha-1)q}ds\right)^{p/q}{\mathbb E}\sup_{t\in [0, T]}\int^T_0 \|Y(t, s)-Y_n(t, s)\|^p_Hds\\
&= \left[\frac{1}{(\alpha-1)q+1}T^{(\alpha-1)q+1}\right]^{p/q}{\mathbb E}\sup_{t\in [0, T]}\int^T_0 \|Y(t, s)-Y_n(t, s)\|^p_Hds
\end{split}
\end{equation}
where $1/p +1/q=1$. On the other hand, we have that for $s\in [0, T]$, 
 \begin{equation}
\label{20/04/09(22)}
\begin{split}
&Y(t, s) - Y_n(t, s)\\
&= \int^0_{-r} [G(t-s+\theta)-G_n(t-s+\theta)] I_1(\theta, s)d\theta +  \int^0_{-r} G_n(t-s+\theta) I_2(\theta, s)d\theta, 
\end{split}
\end{equation}
where 
\[
I_1(\theta, s) = \int^s_0 (s-u)^{-\alpha} [SG(s-u+\cdot)](\theta)B dW(u)\]
and 
\[
I_2(\theta, s) = \int^s_0 (s-u)^{-\alpha}[S(G(s-u+\cdot) -G_n(s-u+\cdot))](\theta)BdW(u).
\]
We claim that 
\begin{equation}
\label{10/05/09(1)}
{\mathbb E}\sup_{t\in [0, T]} \int^T_0 \Big\|\int^0_{-r} [G(t-s+\theta)-G_n(t-s+\theta)] I_1(\theta, s)d\theta\Big\|_H^pds \to 0\hskip 10pt \hbox{as}\hskip 10pt n\to\infty.
\end{equation}
Indeed, by virtue of  H\"older's inequality, Lemma 7.2 in \cite{dpjz92}, Proposition \ref{09/05/09(1213)} and the fact that  $\|G(t)\|\le C(T)$, $\|G_n(t)\|\le C(T)$, $t\in [0, T]$,  for  some $C=C(T)>0$ and large $n\in {\mathbb N}$,  we have by using the well-known Dominated Convergence Theorem that  
 \begin{equation}
\label{20/04/09(23)}
\begin{split}
{\mathbb E}&\sup_{t\in [0, T]} \int^T_0 \Big\|\int^0_{-r} [G(t-s+\theta)-G_n(t-s+\theta)] I_1(\theta, s)d\theta\Big\|_H^pds\\
&\le r^{1/q}{\mathbb E}\sup_{t\in [0, T]}\int^T_0 \int^0_{-r} \Big\|\big[G(t-s+\theta)-G_n(t-s+\theta)\big]I_1(\theta, s)\Big\|^pd\theta ds\\
&\to 0\hskip 10pt \hbox{as}\hskip 10pt n\to\infty.
\end{split}
\end{equation}
 In a similar way, it can be also shown that 
\[
{\mathbb E}\sup_{t\in [0, T]} \int^T_0 \Big\|\int^0_{-r} G_n(t-s+\theta) I_2(\theta, s)d\theta\Big\|^p_Hds
\to 0\hskip 10pt \hbox{as}\hskip 10pt n\to\infty.\]
In order to show $\lim_{n\to\infty}{\mathbb E}\sup_{t\in [0, T]}\|J_2(n, t)\|^p_H=0$, we first notice that 
\begin{equation}
\label{16/05/09(30)}
\begin{split}
J_2(n, t) =&\,\, \int^t_0 (t-s)^{\alpha-1}[G(t-s)-G_n(t-s)]Z(s)ds\\
&\,\, + \int^t_0 (t-s)^{\alpha-1}G_n(t-s)(Z(s)-Z_n(s))ds =: I_3(n, t) +I_4(n, t).
\end{split}
\end{equation}
By virtue of Proposition \ref{09/05/09(1213)},  H\"older inequality and the well-known Dominated Convergence Theorem, there exists a constant $C_T>0$ such that
\[
{\mathbb E}\sup_{t\in [0, T]} \|I_3(n, t)\|^p_Hds\le C_T{\mathbb E}\sup_{t\in [0, T]}\int^t_0 \big\|[G(t-s)-G_n(t-s)]Z(s)\big\|^p_Hds \to 0\hskip 10pt \hbox{as}\hskip 10pt n\to\infty.\]
 On the other hand, by a similar argument as above, we can show that 
\[
{\mathbb E}\sup_{t\in [0, T]} \|I_4(n, t)\|^p_Hds \to 0\hskip 10pt \hbox{as}\hskip 10pt n\to\infty\] by  using the Dominated Convergence Theorem. Therefore, we conclude that the limit   $\lim_{n\to\infty}{\mathbb E}\sup_{t\in [0, T]}\|J_2(n, t)\|^p_H=0$ and the proof of this theorem is complete now.\end{proof}

In the remainder of this section, we shall apply Proposition \ref{10/05/09(2)} to Eq. (\ref{19/11/07(11)}) to study its regulaity property. To this end, we further assume that $A$ generates an analytic semigroup $e^{tA}$, $t\ge 0$, on $H$ and  Tr$\,Q<\infty$. It is essential to establish regularity results for the stochastic convolution  $W_G^B(t)= \int^t_0 G(t-s)BdW(s)$, $t\ge 0$,   
\begin{lemma}
\label{10/05/09(40)} Assume that $a(\cdot)$ in (\ref{09/04/10(1)}) belongs to $L^q([-r, 0]; {\mathbb R}^1)$ for some $q>2$.
 For any $B\in {\mathscr L}_2(K_Q, H)$, let
\[
Z(t) = \int^t_0 G(t-s)BW(s)ds,\hskip 15pt t\ge 0.\]
Then $Z(\cdot) \in C^1([0, \infty); {\mathscr D}(A))$, and moreover
\begin{equation}
\label{21/04/09(1)}
\begin{split}
\frac{dZ(t)}{dt}= AZ(t) + B_1Z(t-r) + \int^0_{-r} a(\theta)B_0Z(t+\theta)d\theta +BW(t)= W^B_G(t),\,\,\,\,t\ge 0.
\end{split}
\end{equation}
\end{lemma}
\begin{proof} By virtue of Proposition \ref{10/05/09(2)}, we know that $W_{G_n}^B(t)\to W^B_G(t)$ almost surely as $n\to\infty$ uniformly with respect to $t$ on any bounded intervals. For any $n\in \rho(A)$, let us consider the following stochastic functional differential equation on the Hilbert space $H$, 
\begin{equation}
\label{03/05/09(1)}
\begin{cases}
dy(t) = A_n y(t)dt +B_1y(t-r)dt + \displaystyle\int^0_{-r} a(\theta)B_0y(t+\theta)d\theta +BdW(t),\,\,\,t\ge 0,\\
 y(t) = 0,\,\,\,\,\,\,-r \le t\le 0,
\end{cases}
\end{equation}
where $A_n$, $n\in \rho(A)$, are the Yosida approximants of $A$. 
According to Theorems 4.1 and 4.2 in \cite{kl08(1)}, it is known that there exists a strong solution for the equation (\ref{03/05/09(1)}). Moreover,  the strong solution is uniquely represented by 
\[ y(t) =W^B_{G_n}(t)= \int^t_0 G_n(t-s)BdW(s)\,\,\,\hbox{for}\,\,\,t\ge 0,\,\,\,\,\,n\in \rho(A),\]
and $y(t)=0$ if $t\in [-r, 0]$. 
Therefore, it follows that for any $t\ge 0$,
\begin{equation}
\label{03/05/09(2)}
\begin{split}
W^B_{G_n}(t) &= A_n \int^t_0 W^B_{G_n}(s)ds +\int^t_0 B_1W^B_{G_n}(s-r)ds\\
&\,\,\,\,\,\,\, + \int^t_0 \int^0_{-r} a(\theta)B_0W^B_{G_n}(s+\theta)d\theta ds + BW(t)\\
&= A_n\int^t_0 W^B_{G_n}(s)ds + B_1\int^{t-r}_0 W^B_{G_n}(s)ds + \int^0_{-r} a(\theta)B_0\int^{t+\theta}_0 W^B_{G_n}(s)ds d\theta +BW(t).
\end{split}
\end{equation}
Let $Z_n(t)= \int^t_0 W^B_{G_n}(s)ds$, $t\ge 0$, then the equality (\ref{03/05/09(2)}) yields that  for any $t\ge 0$,
\[
dZ_n(t) = A_n Z_n(t)dt +B_1Z_n(t-r)dt +\int^0_{-r} a(\theta)B_0Z_n(t+\theta)d\theta dt +BW(t)dt,\]
which immediately  implies that 
\begin{equation}
\label{10/05/09(30)}
Z_n(t) = \int^t_0 G_n(t-s)BW(s)ds,\hskip 15pt t\ge 0.
\end{equation}
In view of Proposition \ref{09/05/09(1213)} and (\ref{10/05/09(30)}), it is easy for ones to deduce  that almost surely
\begin{equation}
\label{10/05/09(200)}
\lim_{n\to\infty} Z_n(t) = \int^t_0 G(t-s)BW(s)ds = Z(t),\hskip 15pt t\ge 0,
\end{equation}
and for any $t\ge 0$,
\begin{equation}
\label{10/05/09(201)}
\lim_{n\to\infty}\Big(B_1Z_n(t-r) + \int^0_{-r} a(\theta)B_0Z_n(t+\theta)d\theta\Big) = B_1Z(t-r) + \int^0_{-r} a(\theta)B_0Z(t+\theta)d\theta.
\end{equation}
Therefore, the relations (\ref{03/05/09(2)}), (\ref{10/05/09(200)}) and (\ref{10/05/09(201)}) together imply that for all $t\ge 0$,
\[\lim_{n\to\infty} AJ_n Z_n(t) =
\lim_{n\to\infty} A_n Z_n(t)  = W^B_G(t)-B_1Z(t-r) -\int^0_{-r} a(\theta)B_0Z(t+\theta)d\theta -BW(t),\]
where $J_n= nR(n, A)$, $n\in \rho(A)$. Since  $J_nx \rightarrow x$ for any $x\in H$ as $n\to\infty$, it follows, in addition to (\ref{10/05/09(200)}), that for any $t\ge 0$,
\[
J_n Z_n(t) = J_n Z(t) +J_n(Z_n(t)- Z(t))\to Z(t)\,\,\,\,\,\hbox{as}\,\,\,n\to\infty.
\]
Due to the closedness of operator $A$, this  implies that $Z(t)\in {\mathscr D}(A)$ and 
\[
AZ(t) = W^B_G(t)-B_1Z(t-r) - \int^0_{-r} a(\theta)B_0Z(t+\theta)d\theta -BW(t),\hskip 20pt t\ge 0,\]
which is exactly the second equality in (\ref{21/04/09(1)}). Finally, as $Z(t)\in {\mathscr D}(A)$, $t\ge 0$, and it is known (cf. Prop. 3.2, \cite{kl08(1)}) that 
\[
\frac{d}{dt}G(t)h = AG(t)h + B_1G(t-r)h + \int^0_{-r} a(\theta)B_0G(t+\theta)d\theta\hskip 10pt\hbox{for any}\hskip 10pt h\in {\mathscr D}(A),\,\,\,\,t\ge 0,\]
we thus obtain the first equality in  (\ref{21/04/09(1)}),
\[
\frac{dZ(t)}{dt} = AZ(t) +B_1Z(t-r) + \int^0_{-r} a(\theta)B_0Z(t+\theta)d\theta +BW(t)\,\,\,\,\hbox{for any}\,\,\,\,t\ge 0.\]
The proof is now complete.
\end{proof}

Now we are ready to present the main results in this section. Recall that the domain ${\mathscr D}(A)\subset H$ is a Banach space under the graph norm $\|h\|_{{\mathscr D}(A)} := \|h\|_H + \|Ah\|_H$, $h\in {\mathscr D}(A)$. 
We also introduce two families of intermediate spaces between ${\mathscr D}(A)$ and $H$, depending on a parameter $\alpha\in (0, 1)$:
\[
{\mathscr D}_A(\alpha, \infty)=\big\{h\in H:\, \|h\|_\alpha :=\sup_{t>0}\|t^{1-\alpha}Ae^{tA}h\|_H<\infty\big\}\]
and 
\[
{\mathscr D}_A(\alpha) =\big\{h\in H:\, \lim_{t\to 0} t^{1-\alpha}Ae^{tA}h=0\big\}.\]
Obviously, it is true that  ${\mathscr D}(A)\subset {\mathscr D}_A(\alpha)\subset {\mathscr D}_A(\alpha, \infty)$. We also know (cf. \cite{pbhb67}, \cite{jlem72}) that ${\mathscr D}_A(\alpha, \infty)$ and ${\mathscr D}_A(\alpha)$ are Banach spaces under the norm $\|\cdot\|_H + \|\cdot\|_\alpha$.

For any $-\infty<a\le b<\infty$ and Banach space $X$ equipped with the norm $\|\cdot\|_X$, we introduce  Banach space
\[
C^\alpha([a, b]; X)= \Big\{u:\, [a, b]\to X;\,\, |u|_\alpha := \sup_{t,\,s\in [a, b],\,t\not=s}\frac{\|u(t)-u(s)\|_X}{|t-s|^\alpha}<\infty\Big\},\]
under the norm $\|u\|_{C^\alpha([a, b]; X)}=\|u\|_{C([a, b]; X)} +|u|_\alpha$, Banach space
\[
C^1([a, b]; X) = \{u:\, [a, b]\to X;\, u,\,u'\in C([a, b]; X)\}\]
under the norm $\|u\|_{C^1([a, b]; X)}=\|u\|_{C([a, b]; X)} +\|u'\|_{C([a, b]; X)}$, and  Banach space 
\[
C^{1, \alpha}([a, b]; X) = \{u:\, [a, b]\to X;\, u\in C^1([a, b]; X),\,u'\in C^\alpha([a, b]; X)\}\]
under the norm $\|u\|_{C^{1, \alpha}([a, b]; X)}=\|u\|_{C([a, b]; X)} +\|u'\|_{C^\alpha([a, b]; X)}$.

\begin{proposition} \label{12/04/10(10)}
Assume that $a(\cdot)$ in (\ref{09/04/10(1)}) belongs to $L^q([-r, 0]; {\mathbb R}^1)$ for some $q>2$.
Let $T\ge 0$ and $B\in {\mathscr L}_2(K_Q, H)$. Then 
\begin{enumerate}
\item[(i)] for any $\alpha\in (0, 1/2)$, the stochastic convolution $W_G^B(t)$ has $\alpha$-H\"older continuous trajectories on $[0, T]$;
\item[(ii)]
 for any $\alpha \in (0, 1/2)$ and $\gamma\in (0, \frac{1}{2}-\alpha)$, $W^B_G(t)\in {\mathscr D}((-A)^\gamma)$, $t\in [0, T]$, and the process $(-A)^\gamma W^B_G(\cdot)$ is $\alpha$-H\"older continuous. 
\end{enumerate}
\end{proposition}
\begin{proof} 
Note that from Lemma \ref{10/05/09(40)}, we have that $W^B_G(t)= dZ(t)/dt$ almost surely where $Z(t)$ is the solution to the initial value problem
\begin{equation}
\label{10/05/09(31)}
\begin{cases}
dZ(t)/dt = AZ(t) +  B_1 Z(t-r) + \displaystyle\int^0_{-r}a(\theta)B_0Z(t+\theta) d\theta +BW(t),\,\,\,\,t\in [0, T],\\
Z(t) = 0,\,\,\,\,t\in [-r, 0].
\end{cases}
\end{equation}
It is known that $W(t)\in C^\alpha([0, T]; H)$ almost surely for $\alpha\in (0, 1/2)$ and thus the conclusion (i) follows easily.

To show (ii), first note (see, e.g., Lemma 1.1 in \cite{al87}) that there is the following inclusion relations
\begin{equation}
\label{10/05/09(42)}
C^{1, \alpha}([0, T]; H)\cap C^\alpha([0, T]; {\mathscr D}(A))\subset C^{1, \alpha-\gamma}([0, T]; {\mathscr D}_A(\gamma, \infty))
\end{equation}
for all $\alpha\in (0, 1)$ and $\gamma\in (0, \alpha)$.
Since $W(t)\in C^\alpha([0, T]; H)$ almost surely for any $\alpha\in (0, 1/2)$, we have by virtue of Proposition \ref{12/04/10(1)} and (\ref{10/05/09(42)}) that
\[
Z(t)\in C^{1, \alpha}([0, T]; H)\cap C^\alpha([0, T]; {\mathscr D}(A))\subset C^{1, \alpha-\gamma}([0, T]; {\mathscr D}_A(\gamma, \infty))\]
almost surely for all $\alpha\in (0, 1/2)$ and $\gamma\in (0, \alpha)$. Thus $W^B_G(t)\in C^{\alpha-\gamma}([0, T]; {\mathscr D}_A(\gamma, \infty))$. Since ${\mathscr D}_A(\gamma, \infty)$ is included in ${\mathscr D}((-A)^{\gamma-\varepsilon})$ (cf. Prop. A.13 in \cite{dpjz92}) for sufficiently small $\varepsilon>0$, the desired conclusion (ii) thus follows. The proof is  complete now.
\end{proof}

Combining Lemma \ref{12/04/10(10)} and Proposition \ref{12/04/10(1)}, we obtain the following regularity results for the solution of Eq. (\ref{19/11/07(11)}).
\begin{theorem}
Let $T\ge 0$ and $B\in {\mathscr L}_2(K_Q, H)$. 
Assume that $a(\cdot)$ in (\ref{09/04/10(1)}) belongs to $L^q([-r, 0]; {\mathbb R}^1)$ for some $q>2$, $\phi_0\in {\mathscr D}(A)$ and $\phi_1\in C^\alpha([-r, 0]; {\mathscr D}(A))$, $\alpha\in (0, 1/2)$, such that 
\[
A\phi_0 + B_1\phi_1(-r) + \int^0_{-r} a(\theta)B_0\phi_1(\theta)d\theta\in {\mathscr D}_A(\alpha, \infty).\]
Then the solution $y(t)$ of Eq. (\ref{19/11/07(11)}) has $\alpha$-H\"older continuous trajectories on $[0, T]$. Moreover, for all $\gamma \in (0, \frac{1}{2}-\alpha)$, the solution $y(t)\in {\mathscr D}((-A)^\gamma)$ and the process $(-A)^\gamma y(t)$ is $\alpha$-H\"older continuous on $t\in [0, T]$. 
\end{theorem}

\section{Burkholder-Davis-Gundy Inequality}

In this section, we shall establish in Theorem  \ref{16/05/09(3)}  a retarded version of the well-known Burkholder-Davis-Gundy type of inequality for the stochastic convolution  $W_G^B(t) = \int^t_0 G(t-s)B(s)dW(s)$ for any process $B\in {\cal U}^2([0, T]; {\mathscr L}_2(K_Q, H))$, $T\ge 0$. It is worth pointing out that a similar inequality of this kind was established in \cite{kl08(1)} under the additional restriction that $A$ generates a pseudo contraction semigroup $e^{tA}$, $t\ge 0$, in the sense that  $\|e^{tA}\|\le e^{\alpha t}$ holds true for some constant $\alpha\in {\mathbb R}^1$ and all $t\ge 0$. By using a factorization method, it is possible for one to remove this restriction to establish a similar inequality. To see this, we first impose, motivated by Lemma \ref{12/04/2010(2)}, the following conditions as in  \cite{kl08(1)} on the delay opeartor $F$  in (\ref{10/02/08(1)}).

Let $1\le p<\infty$ and assume that the operator $F$  in (\ref{10/02/08(1)}) permits a bounded linear extension $F:\, L^{p}([-r, T]; H)\to L^{p}([0, T]; H)$ which is defined by $(Fy)(t)=Fy_t$, $y\in L^{p}([-r, T]; H)$, i.e.,  there exists a real number $M_p >0$ such that 
\begin{equation}
\label{10/02/08(1032)}
\int^T_0 \|(Fy)(t)\|^{p}_Hdt \le M_p \int^T_{-r} \|y(t)\|^{p}_Hdt\hskip 15pt \hbox{for any}\hskip 15pt y\in L^{p}([-r, T]; H).
\end{equation} 

As a consequence, it is valid that under the condition (\ref{10/02/08(1032)}), the structure operator $S$ introduced in  (\ref{28/06/06(101)}) (resp. $S$ in 
(\ref{27/06/06(302)})) can be extended, similarly to (\ref{12/04/10(3)}), to a linear and bounded operator from the space $L^{p}([-r, 0]; H)$ (resp. $L^{p}([-r, 0]; {\mathscr L}(H))$) into itself. Moreover, 
\begin{equation}
\label{12/04/10(303)}
\begin{split}
\int^0_{-r} \|S\varphi(\theta)\|^{p}_Hd\theta \le M_p \int^0_{-r} \|\varphi(\theta)\|^{p}_Hd\theta,\hskip 15pt \forall\, \varphi(\cdot)\in L^{p}([-r, 0]; H).
\end{split}
\end{equation}

\begin{theorem} \label{16/05/09(3)}
{\bf (Burkholder-Davis-Gundy Inequality)} 
Let $p>2$, $T\ge 0$ and $B(t)$ be an ${\mathscr L}_2(K_Q, H)$-valued, ${\mathscr F}_t$-adapted process such that 
\begin{equation}
\label{01/03/09(1)}
{\mathbb E}\int^T_0 Tr[B(t)QB(t)^*]^{p/2}dt<\infty.
\end{equation}
Suppose further that the inequality (\ref{10/02/08(1032)}) holds, then there exists a number $C=C(T)>0$ such that 
\begin{equation}
\label{09/05/09(2)}
{\mathbb E}\Big(\sup_{t\in [0, T]}\Big\|\int^t_0 G(t-s)B(s)dW(s)\Big\|^{p}_H\Big)\le C(T) {\mathbb E}\int^T_0 Tr[B(t)QB(t)^*]^{p/2}dt.
\end{equation}
\end{theorem}
\begin{proof}
Let $\alpha\in (0, \frac{p-2}{2p})$ and it is known in (\ref{20/04/09(1)}) that  
 \begin{equation}
\label{01/03/09(2)}
\begin{split}
&W_G^B(t)\cr 
&=  \frac{\sin \pi\alpha}{\pi}\Big\{\int^t_0 (t-s)^{\alpha-1}\int^s_0 (s-u)^{-\alpha}\int^0_{-r} G(t-s+\theta)[SG(s-u+\cdot)](\theta) B(u)d\theta dW(u)ds\cr
&\,\,\,\,\,\,\,\, +\int^t_0 (t-s)^{\alpha-1}G(t-s)\int^s_0 G(s-u)(s-u)^{-\alpha}B(u)dW(u)ds\Big\}\cr
&= \frac{\sin \pi\alpha}{\pi}(I_1(t) +I_2(t)).
\end{split}
\end{equation}
Firstly,  let us estimate the term $I_1(t)$. To this end, we rewrite $I_1(t)$ as
\[
I_1(t) = \int^t_0 (t-s)^{\alpha-1}Z(t, s)ds,\hskip 20pt t\in [0, T],\]
where 
\[
Z(t,s) = \int^s_0 (s-u)^{-\alpha}\int^0_{-r}G(t-s+\theta)[SG(s-u+\cdot)](\theta)B(u)d\theta dW(u),\hskip 15pt s\in [0, T].\]
Recall the well-known Young inequality: for any $p>1$, 
 \begin{equation}
\label{01/03/09(30)}
\Big| \int^T_0 (u\ast v)(t)dt\Big|^p \le T^{p-1}\Big(\int^T_0 |u(t)|dt\Big)^p \int^T_0 |v(t)|^pdt
\end{equation}
holds for all $u\in L^1([0, T]; {\mathbb R}^1)$ and $v\in L^p([0, T]; {\mathbb R}^1)$ where $u\ast v$ is the convolution of the real-valued functions $u$ and $v$. Since $\alpha>0$,
it follows by virtue of  (\ref{01/03/09(30)}) that  
 \begin{equation}
\label{01/03/09(3025)}
{\mathbb E}\sup_{t\in [0, T]}\|I_1(t)\|^{p}_H \le C_{1, T} {\mathbb E}\sup_{t\in [0, T]}\int^T_0 \|Z(t, s)\|^{p}_H ds
\end{equation}
where 
\[
C_{1, T} = T^{p-1}\left(\int^T_0 s^{\alpha -1}ds\right)^p = \frac{1}{\alpha}T^{p+\alpha-1}>0.\]
Since $\|G(t)\|\le  C_{2, T}$, $t\in [0, T]$, for some number $C_{2, T}>0$, one  has that for any $t\in [0, T]$,
 \begin{equation}
\label{01/03/09(302)}
\begin{split}
\int^T_0 &\|Z(t, s)\|^{p}_Hds\\
 &= \int^T_0\Big\|\int^0_{-r} G(t-s+\theta)\int^s_0 (s-u)^{-\alpha}[SG(s-u+\cdot)](\theta)B(u)dW(u)d\theta\Big\|^{p}_Hds\\
&\le C^p_{2, T}\int^T_0\Big(\int^0_{-r}\Big\|\int^s_0 (s-u)^{-\alpha}[SG(s-u+\cdot)](\theta)B(u)dW(u)\Big\|_Hd\theta \Big)^{p}ds
\end{split}
\end{equation}
which, by H\"older inequality, immediately implies that
 \begin{equation}
\label{01/03/09(3026)}
\begin{split}
{\mathbb E}&\sup_{t\in [0, T]}\int^T_0\|Z(t, s)\|^{p}_Hds\\ &\le r^{1-1/p}C^p_{2, T}\int^T_0\int^0_{-r} {\mathbb E}\Big\|\int^s_{0} (s-u)^{-\alpha}[SG(s-u+\cdot)](\theta)B(u)dW(u)\Big\|^{p}_Hd\theta ds\\
& = r^{1-1/p}C^p_{2, T}\int^T_0 J(s)ds,
\end{split}
\end{equation}
where 
\[
J(s) = \int^0_{-r} {\mathbb E}\Big\|\int^s_{0} (s-u)^{-\alpha}[SG(s-u+\cdot)](\theta)B(u)dW(u)\Big\|^{p}_Hd\theta.\]
As for the term $J(\cdot)$, by using  Lemma 7.2 in \cite{dpjz92} and H\"older inequality we can obtain that for some number $C_{3, T}>0$ and any $s\in [0, T]$,
 \begin{equation}
\label{01/03/09(30268)}
\begin{split}
&J(s)\\
 & \le C_{3, T}\int^0_{-r} {\mathbb E}\Big( \int^s_0 (s-u)^{-2\alpha}Tr[SG(s-u+\cdot)(\theta)B(u)Q(SG(s-u+\cdot)(\theta)B(u))^*]du\Big)^{p/2}d\theta\\
&= C_{3, T}\int^0_{-r} \Big(\int^s_{0} (s-u)^{- \frac{2\alpha p}{p-2}}du\Big)^{\frac{p-2}{2}}\\
&\,\,\,\,\,\,\,\cdot  {\mathbb E} \Big( \int^s_0 Tr[SG(s-u+\cdot)(\theta)B(u)Q(SG(s-u+\cdot)(\theta)B(u))^*]^{p/2}du\Big)d\theta.
\end{split}
\end{equation}
Since $0<\alpha<\frac{p-2}{2p}$, it is easy to see that $1- \frac{2\alpha p}{p-2}>0$ and  let
\begin{equation}
\label{15/05/09(20)}
\begin{split}
C_{4, T} = \left(\int^T_0 (T-s)^{-\frac{2\alpha p}{p-2}}ds\right)^{\frac{p-2}{2}} = \left(\frac{p-2}{p-2 -2\alpha p}\cdot T^{1-\frac{2\alpha p}{p-2}}\right)^{\frac{p-2}{2}}<\infty.
\end{split}
\end{equation}
Then (\ref{12/04/10(303)}), (\ref{01/03/09(3026)}), (\ref{01/03/09(30268)}) and (\ref{15/05/09(20)}) together imply that
 \begin{equation}
\label{01/03/09(30278)}
\begin{split}
{\mathbb E}&\sup_{t\in [0, T]}\int^T_0\|Z(t, s)\|^{p}_Hds\\ & \le r^{1-1/p}C^p_{2, T}C_{3, T} C_{4, T}{\mathbb E}\int^T_0\int^0_{-r} \int^s_{0}  Tr[SG(s-u+\cdot)(\theta)B(u)\\
&\,\,\,\,\,\,\,\cdot Q(SG(s-u+\cdot)(\theta)B(u))^*]^{p/2}du d\theta ds\\
& \le r^{1-1/p}C^p_{2, T}C_{3, T} C_{4, T}\int^T_0\int^s_0 \int^0_{-r} {\mathbb E}\big\|[SG(s-u+\cdot)](\theta)\big\|^{p} Tr[B(u)QB(u)^*]^{p/2}d\theta du ds\\
&\le r^{1-1/p}C^p_{2, T}C_{3, T}C_{4, T}M_p\int^T_0 \int^s_0 \int^0_{-r} {\mathbb E}\|G(s-u+\theta)\|^{p}  Tr[B(u)QB(u)^*]^{p/2}d\theta duds\\
&\le  C_{5, T}{\mathbb E}\int^T_0 Tr[B(u)QB(u)^*]^{p/2}du<\infty,
\end{split}
\end{equation}
where $C_{5, T}= r^{1-1/p} C^{2p}_{2, T}C_{3, T} C_{4, T}M_p T>0$. Therefore, the relations (\ref{01/03/09(30278)}) and (\ref{01/03/09(3025)}) immediately yield that  
\begin{equation}
\label{19/04/09(10)}
{\mathbb E}\sup_{t\in [0, T]}\|I_1(t)\|^{p}_H \le  C_{1, T}C_{5, T}{\mathbb E}\int^T_0 Tr[B(u)QB(u)^*]^{p/2}du.
\end{equation}
In a similar way, we can show that there exists a real number $C_{6, T}>0$ such that the following inequality holds:
\begin{equation}
\label{19/04/09(11)}
{\mathbb E}\sup_{t\in [0, T]}\|I_2(t)\|^{p}_H \le C_{6, T}{\mathbb E}\int^T_0 Tr[B(u)QB(u)^*]^{p/2}du.
\end{equation}
The inequalities (\ref{19/04/09(10)}) and (\ref{19/04/09(11)}), in addition to  (\ref{01/03/09(2)}), imply the desired result (\ref{09/05/09(2)}). The proof is thus complete.
 \end{proof}

\section{Appendix}

In this appendix, we shall recall and establish some regularity properties for a class of  deterministic functional differential equations on the Hilbert space $H$. Firstly, consider the deterministic equation without time delays
\begin{equation}
\label{06/04/10(1)}
\begin{cases}
dy(t)/dt = Ay(t) + f(t),\hskip 15pt t\in [0, T],\hskip 10pt T\ge 0,\\
y(0) = \phi_0\in H,
\end{cases}
\end{equation} 
where $A$ generates an analytic semigroup $e^{tA}$, $t\ge 0$ and $f\in L^1([0, T]; H)$. The 
following regularity result is well established and its proofs are referred to the existing literature, e.g., \cite{jlem72} or \cite{ap83}.
\begin{proposition}
\label{07/04/01(10)}
Let $\alpha\in (0, 1)$ and suppose that 
\[
\phi_0\in {\mathscr D}(A),\hskip 15pt f\in C^\alpha([0, T]; H),\hskip 15pt A\phi_0 + f(0)\in {\mathscr D}_A(\alpha, \infty),\]
then the function
\[
y(t) = e^{tA}\phi_0 + \int^t_0 e^{(t-s)A}f(s)ds\in C^1([0, T]; H)\cap C([0, T]; {\mathscr D}(A)),\hskip 15pt t\in [0, T],\]
is the unique (classical) solution of  (\ref{06/04/10(1)}). Moreover, we have 
\begin{equation}
\label{12/04/10(20)}
y\in C^{1, \alpha}([0, T]; H)\cap C^\alpha([0, T]; {\mathscr D}(A)),
\end{equation}
and the estimate
\begin{equation}
\label{12/04/10(21)}
\begin{split}
\max\{\|y\|_{C^{1, \alpha}([0, T]; H)}, \|y\|_{C^{\alpha}([0, T]; {\mathscr D}(A))}\}
\le C\{\|f\|_{C^\alpha([0, T]; H)} + \|A\phi_0 +f(0)\|_{{\mathscr D}_A(\alpha, \infty)} + \|\phi_0\|_H\},
\end{split}
\end{equation}
where $C= C({\alpha, T})$ is some positive number.
\end{proposition}

 The main objective in the appendix is to establish the existence and uniqueness of (classical) solutions for a class of linear functional differential equations (\ref{07/04/10(1)}) below. 

Let $T\ge 0$ and $r>0$ and $A$ generate an analytic semigroup $e^{tA}$, $t\ge 0$, on $H$ and $B_i\in {\mathscr L}(H)$, $i=0,\,1$. Suppose that  $f\in L^1([0, T]; H)$, $a\in L^1([-r, 0]; {\mathbb R}^1)$ and $\Phi=(\phi_0, \phi_1)\in {\cal H}$. Consider the following linear differential equation with time delays
\begin{equation}
\label{07/04/10(1)}
\begin{cases}
dy(t)/dt= Ay(t) + B_1y(t-r) + \displaystyle\int^0_{-r} a(\theta)B_0y(t+\theta)d\theta +f(t),\hskip 15pt 0\le t\le T,\\
y(0) =\phi_0,\,\,\,y(t)=\phi_1(t),\,\,\,t\in [-r, 0),
\end{cases}
\end{equation}

\begin{definition}
\rm A function $y: [-r, T]\to H$ which is differentiable on $[0, T]$ almost everywhere is called a  {\it (classical) solution\/} of the initial value problem (\ref{07/04/10(1)}) on $[-r, T]$ if $y\in C([0, T]; {\mathscr D}(A))\cap C^1([0, T]; H)$ and the equation  (\ref{07/04/10(1)}) is verified.
\end{definition}

\begin{proposition}
\label{12/04/10(1)}
Suppose  that $f\in C^\alpha([0, T]; H)$, $\phi_0\in {\mathscr D}(A)$ and $\phi_1\in C^\alpha([-r, 0]; {\mathscr D}(A))$ such that 
\[
A\phi_0 + B_1\phi_1(-r) + \int^0_{-r} a(\theta)B_0\phi_1(\theta)d\theta + f(0)\in {\mathscr D}_A(\alpha, \infty),\]
then there is a unique (classical) solution $y$ of the problem (\ref{07/04/10(1)}) with 
\[
y\in C^\alpha([0, T]; {\mathscr D}(A))\cap C^{1, \alpha}([0, T]; H).\]
Moreover, there exists a number $C = C(T, \alpha, a, B_0)>0$ such that 
\begin{equation}
\label{12/04/10(70)}
\begin{split}
\max\{\|y\|_{C^\alpha([0, T]; {\mathscr D}(A))}, &\, \|y\|_{C^{1, \alpha}([0, T]; H)}\}
\le C\Big(\|f\|_{C^\alpha([0, T]; H)} + \|\phi_1\|_{C^\alpha([-r, 0]; {\mathscr D}(A))} + \|\phi_0\|_H\\
&\,\,\,\,\, + \Big\|A\phi_0 + B_1\phi_1(-r) + \int^0_{-r} a(\theta)B_0\phi_1(\theta)d\theta + f(0)\Big\|_{{\mathscr D}_A(\alpha, \infty)}\Big).
\end{split}
\end{equation}
\end{proposition}
\begin{proof}
For any $\delta>0$, consider the following closed subset of $C^\alpha([0, \delta]; {\mathscr D}(A))$: 
\[
E = \{\bar y\in C^\alpha([0, \delta]; {\mathscr D}(A)):\,\, \bar y(0)=\phi_0\}.\]
W can associate to each $\bar y\in E$ a function $y\in C^\alpha([-r, \delta]; {\mathscr D}(A))$ by
\begin{equation}
\label{12/04/10(4)}
y(t) =\begin{cases}
\phi_1(t),&\hskip 20pt -r\le t<0,\\
\bar y(t),&\hskip 20pt t\in [0, \delta],
\end{cases}
\end{equation}
and define the mapping $\Xi$ for any $t\in [0, \delta]$ by 
\[
(\Xi\bar y)(t) = e^{tA}\phi_0 + \int^t_0 e^{(t-s)A}\Big[B_1\phi_1(s-r) + \int^0_{-r} a(\theta)B_0 y(s+\theta)d\theta +f(s)\Big]ds.\]
As $\phi_1\in C^\alpha([-r, 0]; {\mathscr D}(A))$ and $B_1\in {\mathscr L}(H)$, it is immediate that  $B_1\phi_1(\cdot-r)\in C^\alpha([0, \delta]; H)$. On the other hand, for any $y\in C^\alpha([-r, \delta]; {\mathscr D}(A))$, it is easy to see that $\int^0_{-r} a(\theta)B_0 y(s+\theta)d\theta\in C^\alpha([0, \delta]; H)$, $s\in [0, \delta]$, and moreover
\begin{equation}
\label{14/04/10(70)}
\Big\|\int^0_{-r} a(\theta)B_0 y(\cdot+\theta)d\theta\Big\|_{C^\alpha([0, \delta]; H)}\le \|B_0\|\|a\|_{L^1([-r, 0]; {\mathbb R}^1)}\|y\|_{C^\alpha([-r, \delta]; {\mathscr D}(A))}.
\end{equation}
This means, in addition to $f\in C^\alpha([0, \delta]; H)$,  that the function 
\begin{equation}
\label{12/04/10(31)}
B_1\phi_1(\cdot-r) + \int^0_{-r} a(\theta)B_0 y(\cdot+\theta)d\theta +f(\cdot)\in C^\alpha([0, \delta]; H).
\end{equation}
 Taking into account the fact that $\phi_0\in {\mathscr D}(A)$ and 
\[
A\phi_0 + B_1\phi_1(-r) + \int^0_{-r} a(\theta)B_0 y(\theta)d\theta +f(0)\in {\mathscr D}_A(\alpha, \infty),\]
we deduce from Proposition \ref{07/04/01(10)} that there exists a solution of the equation (\ref{07/04/10(1)}) in $[0, \delta]$ if and only if there exists an element $\bar y\in E$ such that 
\[
\Xi\bar y =\bar y,\]
and in this case it is true that $y\in C^{1, \alpha}([0, \delta]; H)$ by virtue of (\ref{12/04/10(20)}). 

Next we shall prove that the map $\Xi$ is a contraction on $E$ for properly chosen $\delta>0$. To this end, for any $\bar y_i\in E$ define $y_i$, $i=1,\,2$, according to  (\ref{12/04/10(4)}). Then we have for $t\in [0, \delta]$ that 
\begin{equation}
\label{13/04/10(1)}
\Xi\bar y_1(t) - \Xi\bar y_2(t) = \int^t_0 e^{(t-s)A}\int^0_{-r} a(\theta)B_0(y_1(s+\theta)-y_2(s+\theta))d\theta ds.
\end{equation}
By analogy with (\ref{14/04/10(70)}), it is not difficult to see that
\[
\int^0_{-r} a(\theta)B_0(y_1(\cdot +\theta)-y_2(\cdot +\theta))d\theta\in C^\alpha([0, \delta]; H),\]
and thus applying (\ref{12/04/10(21)}) to (\ref{13/04/10(1)}) yields that
\begin{equation}
\label{12/04/10(26)}
\|\Xi\bar y_1 -\Xi\bar y_2\|_{C^\alpha([0, \delta]; {\mathscr D}(A)}\le C\Big\|\int^0_{-r} a(\theta)B_0(y_1(\cdot+\theta)-y_2(\cdot + \theta))d\theta\Big\|_{C^\alpha([0, \delta]; H)}
\end{equation}
where $C>0$ is the constant given in (\ref{12/04/10(21)}). On the other hand, let $\delta\in (0, r)$ and as $y_1(t)-y_2(t)=0$ for $t\in [-r, 0]$, it is easy to see that 
\begin{equation}
\label{12/04/10(75)}
\Big\|\int^0_{-r} a(\theta)B_0(y_1(\cdot+\theta)-y_2(\cdot+\theta))d\theta\Big\|_{C^\alpha([0, \delta]; H)}\le \|B_0\|\|a\|_{L^1([-\delta, 0]; {\mathbb R}^1)}\|y_1-y_2\|_{C^\alpha([0, \delta]; {\mathscr D}(A))},
\end{equation}
which, in addition to  (\ref{12/04/10(26)}), immediately yields that 
\[
\|\Xi\bar y_1 -\Xi\bar y_2\|_{C^\alpha([0, \delta]; {\mathscr D}(A))}\le C\|B_0\|\|a\|_{L^1([-\delta, 0]; {\mathbb R}^1)} \|\bar y_1-\bar y_2\|_{C^\alpha([0, \delta]; {\mathscr D}(A))}.
\]
This implies that $\Xi$ is a contraction if we choose $\delta>0$ small enough ($\delta<r$) such that 
\[
 C\|B_0\|\|a\|_{L^1([-\delta, 0]; {\mathbb R}^1)} <1,\]
and thus guarantee the existence of a unique solution $y\in C^\alpha([0, \delta]; {\mathscr D}(A))$ of (\ref{07/04/10(1)}) in $[0, \delta]$.

Next we shall show the estimate (\ref{12/04/10(70)}). Firstly, it is easy to see that for $t\in [0, \delta]$, there is 
\[
\bar y(t) = e^{tA}\phi_0 + \int^t_0 e^{(t-s)A} \Big[B_1\phi_1(s-r) +\int^0_{-r} a(\theta)B_0y(s+\theta)d\theta +f(s)\Big]ds,\]
and so, by virtue of (\ref{12/04/10(21)}), it follows that 
\begin{equation}
\label{12/04/10(80)}
\begin{split}
\|\bar y\|_{C^\alpha([0, \delta]; {\mathscr D}(A))}&\le C\Big\{\|A\phi_1\|_{C^\alpha([-r, \delta-r]; H)} + \Big\|\int^0_{-r} a(\theta)B_0 y(\cdot +\theta)d\theta\Big\|_{C^\alpha([0, \delta]; H)} + \|\phi_0\|_H\\
&\,\,\,\,\,\, + \|f\|_{C^\alpha([0, \delta]; H)} + \Big\|A\phi_0 + B_1\phi_1(-r) +\int^0_{-r} a(\theta)B_0\phi_1(\theta)d\theta +f(0)\Big\|_{{\mathscr D}_A(\alpha, \infty)}\Big\}.
\end{split}
\end{equation}
Note that we have the estimates
\begin{equation}
\label{12/04/10(81)}
\|A\phi_1\|_{C^\alpha([-r, \delta-r]; H)} \le \|\phi_1\|_{C^\alpha([-r, 0]; {\mathscr D}(A))},
\end{equation} 
and 
\begin{equation}
\label{12/04/10(83)}
\begin{split}
\Big\|\int^0_{-r} a(\theta)B_0 y(\cdot+\theta)\Big\|_{C^\alpha([0, \delta]; H)} \le &\,\,\|B_0\|\Big\{
\|a\|_{L^1([-r, 0]; {\mathbb R}^1)}\|\phi_1\|_{C^\alpha([-r, 0]; {\mathscr D}(A))}\\
&\,\, + \|a\|_{L^1([-\delta, 0]; {\mathbb R}^1)}\|\bar y\|_{C^\alpha([0, \delta]; {\mathscr D}(A))}\Big\}.
\end{split}
\end{equation}
Thus, by substituting (\ref{12/04/10(81)}) and (\ref{12/04/10(83)}) into   (\ref{12/04/10(80)}), we obtain that 
\begin{equation}
\label{12/04/10(84)}
\begin{split}
\|\bar y\|_{C^\alpha([0, \delta]; {\mathscr D}(A))} &\le \frac{C(1+ \|B_0\|\|a\|_{L^1([-r, 0]; 
{\mathbb R}^1)})}{1-C\|B_0\|\|a\|_{L^1([-\delta, 0]; {\mathbb R}^1)}}\Big\{\|f\|_{C^\alpha([0, \delta]; H)} + 
\|\phi_1\|_{C^\alpha([0, \delta]; {\mathscr D}(A))} + \|\phi_0\|_H\\
&\,\,\,\,\, + \Big\|A\phi_0 + B_1\phi_1(-r) + \int^0_{-r} a(\theta)B_0\phi_1(\theta)d\theta + f(0)\Big\|_{{\mathscr D}_A(\alpha, \infty)}\Big\}.
\end{split}
\end{equation}
On the other hand, let us observe that 
\begin{equation}
\label{12/04/10(86)}
\begin{split}
\|\bar y&\|_{C^{1, \alpha}([0, \delta]; H)} = \|\bar y\|_{C([0, \delta]; H)} + \|\bar y'\|_{C^\alpha([0, \delta]; H)}\\
&\le \|\bar y\|_{C^\alpha([0, \delta]; {\mathscr D}(A))} + \Big\|A\bar y(\cdot) + B_1\phi_1(\cdot -r) + \int^0_{-r} a(\theta)B_0y(\cdot +\theta)d\theta + f(\cdot)\Big\|_{C^\alpha([0, \delta]; H)}\\
& \le 2\|\bar y\|_{C^\alpha([0, \delta]; {\mathscr D}(A))} + \Big\|B_1\phi_1(\cdot -r) + \int^0_{-r} a(\theta)B_0y(\cdot +\theta)d\theta + f(\cdot)\Big\|_{C^\alpha([0, \delta]; H)}.
\end{split}
\end{equation}
However, we know from (\ref{12/04/10(83)}) that  
\begin{equation}
\label{12/04/10(88)}
\begin{split}
\Big\|B_1\phi_1(\cdot -r) + &\,\int^0_{-r} a(\theta)B_0y(\cdot +\theta)d\theta + f(\cdot)\Big\|_{C^\alpha([0, \delta]; H)}\\
&\le \|B_1\|\|\phi_1\|_{C^\alpha([0, \delta]; {\mathscr D}(A))} + \|B_0\|\Big\{\|a\|_{L^1([-r, 0]; {\mathbb R}^1)} \|\phi_1\|_{C^\alpha([0, \delta]; {\mathscr D}(A))}\\
&\,\,\,\,\,\,\,\,\, + \|a\|_{L^1([-\delta, 0]; {\mathbb R}^1)} \|\bar y\|_{C^\alpha([0, \delta]; {\mathscr D}(A))}\Big\} + \|f\|_{C^\alpha([0, \delta]; H)}.
\end{split}
\end{equation}
Hence, (\ref{12/04/10(84)}), (\ref{12/04/10(86)}) and (\ref{12/04/10(88)}) together yield the desired result (\ref{12/04/10(70)}) on $[0, \delta]$.

The existence of a solution on the whole interval $[0, T]$ and the associated relation  (\ref{12/04/10(70)}) can be similarly established by repeating the above arguments on $[\delta, 2\delta]$, $[2\delta, 3\delta]$, $\cdots$, and the proof is thus complete.
\end{proof}

\vskip 20pt
\noindent {\bf \Large Acknowledgement}
\vskip 20pt
\noindent The author wishes to  thank  the  referee and associate editor for making some helpful suggestions, especially those in association with the proofs in Proposition 5.1, to improve the presentation of this paper. Also, the author gratefully acknowledges some financial supports from NSFC Research Grant, Ref. 10971041.

\end{document}